\documentclass[11pt,a4paper]{amsart}
\usepackage[margin=25mm]{geometry}
\usepackage[numbers]{natbib}
\usepackage{hyperref}

\usepackage{amssymb,amsmath}
\usepackage{amsthm}
\usepackage{bbm, bm}
\usepackage{stmaryrd}
\usepackage{color}
\usepackage{graphicx}
\usepackage{caption}
\usepackage{float}
\usepackage{wrapfig}
\usepackage{subcaption}
\usepackage{amsaddr}
\usepackage[mathscr]{eucal}
\usepackage{enumerate}
\usepackage{orcidlink}
\usepackage{lineno}

\usepackage{abstract}
\usepackage{twemojis}

\input xy
\xyoption{all} 

\numberwithin{equation}{section}

\newtheorem{theorem}{Theorem}[section]
\newtheorem{corollary}{Corollary}[section]
\newtheorem{lemma}{Lemma}[section]
\newtheorem{proposition}{Proposition}[section]

\newtheorem{observation}{Observation}[section]
\theoremstyle{definition}
\newtheorem{definition}{Definition}[section]
\newtheorem{example}{Example}[section]

\newenvironment{remark}[1][Remark.]{\begin{trivlist}
\item[\hskip \labelsep {\bfseries #1}]  }{ \end{trivlist}}

\newcommand{\Diffeo}{\mathrel{\raisebox{-.25ex}{$\xrightarrow{\sim}$}}}
\newcommand{\Dleft}{[\hspace{-1.5pt}[}
\newcommand{\Dright}{]\hspace{-1.5pt}]}
\newcommand{\SN}[1]{\Dleft #1 \Dright}

\newcommand{\rmi}{ \textnormal{i}}
\newcommand{\rmd}{\textnormal{d}}

\DeclareMathOperator{\Vect}{Vect}
\DeclareMathOperator{\Rank}{Rank}

\DeclareMathOperator{\Span}{Span}

\DeclareMathOperator{\Der}{Der}

\newcommand{\catname}[1]{\textnormal{\texttt{#1}}}

\font\black=cmbx10 \font\sblack=cmbx7 \font\ssblack=cmbx5 \font\blackital=cmmib10  \skewchar\blackital='177
\font\sblackital=cmmib7 \skewchar\sblackital='177 \font\ssblackital=cmmib5 \skewchar\ssblackital='177
\font\sanss=cmss10 \font\ssanss=cmss8 
\font\sssanss=cmss8 scaled 600 \font\blackboard=msbm10 \font\sblackboard=msbm7 \font\ssblackboard=msbm5
\font\caligr=eusm10 \font\scaligr=eusm7 \font\sscaligr=eusm5  \font\fraktur=eufm10
\font\sfraktur=eufm7 \font\ssfraktur=eufm5 
\font\bsymb=cmsy10 scaled\magstep2
\def\all#1{\setbox0=\hbox{\lower1.5pt\hbox{\bsymb
       \char"38}}\setbox1=\hbox{$_{#1}$} \box0\lower2pt\box1\;}
\def\exi#1{\setbox0=\hbox{\lower1.5pt\hbox{\bsymb \char"39}}
       \setbox1=\hbox{$_{#1}$} \box0\lower2pt\box1\;}

\def\tx#1{{\fam0\relax#1}}

\newfam\bifam
\textfont\bifam=\blackital \scriptfont\bifam=\sblackital \scriptscriptfont\bifam=\ssblackital

\newfam\blfam
\textfont\blfam=\black \scriptfont\blfam=\sblack \scriptscriptfont\blfam=\ssblack

\newfam\bbfam
\textfont\bbfam=\blackboard \scriptfont\bbfam=\sblackboard \scriptscriptfont\bbfam=\ssblackboard

\newfam\ssfam
\textfont\ssfam=\sanss \scriptfont\ssfam=\ssanss \scriptscriptfont\ssfam=\sssanss
\def\sss#1{{\fam\ssfam\relax#1}}

\newfam\clfam
\textfont\clfam=\caligr \scriptfont\clfam=\scaligr \scriptscriptfont\clfam=\sscaligr

\newfam\frfam
\textfont\frfam=\fraktur \scriptfont\frfam=\sfraktur \scriptscriptfont\frfam=\ssfraktur

\def\hpb#1{\setbox0=\hbox{${#1}$}
    \copy0 \kern-\wd0 \kern.2pt \box0}
\def\vpb#1{\setbox0=\hbox{${#1}$}
    \copy0 \kern-\wd0 \raise.08pt \box0}

\def\pmb#1{\setbox0\hbox{${#1}$} \copy0 \kern-\wd0 \kern.2pt \box0}
\def\pmbb#1{\setbox0\hbox{${#1}$} \copy0 \kern-\wd0
      \kern.2pt \copy0 \kern-\wd0 \kern.2pt \box0}
\def\pmbbb#1{\setbox0\hbox{${#1}$} \copy0 \kern-\wd0
      \kern.2pt \copy0 \kern-\wd0 \kern.2pt
    \copy0 \kern-\wd0 \kern.2pt \box0}
\def\pmxb#1{\setbox0\hbox{${#1}$} \copy0 \kern-\wd0
      \kern.2pt \copy0 \kern-\wd0 \kern.2pt
      \copy0 \kern-\wd0 \kern.2pt \copy0 \kern-\wd0 \kern.2pt \box0}
\def\pmxbb#1{\setbox0\hbox{${#1}$} \copy0 \kern-\wd0 \kern.2pt
      \copy0 \kern-\wd0 \kern.2pt
      \copy0 \kern-\wd0 \kern.2pt \copy0 \kern-\wd0 \kern.2pt
      \copy0 \kern-\wd0 \kern.2pt \box0}

\mathchardef\za="710B  
\mathchardef\zb="710C  
\mathchardef\zg="710D  
\mathchardef\zd="710E  
\mathchardef\zve="710F 
\mathchardef\zz="7110  
\mathchardef\zh="7111  
\mathchardef\zvy="7112 
\mathchardef\zi="7113  
\mathchardef\zk="7114  
\mathchardef\zl="7115  
\mathchardef\zm="7116  
\mathchardef\zn="7117  
\mathchardef\zx="7118  
\mathchardef\zp="7119  
\mathchardef\zr="711A  
\mathchardef\zs="711B  
\mathchardef\zt="711C  
\mathchardef\zu="711D  
\mathchardef\zvf="711E 
\mathchardef\zq="711F  
\mathchardef\zc="7120  
\mathchardef\zw="7121  
\mathchardef\ze="7122  
\mathchardef\zy="7123  
\mathchardef\zf="7124  
\mathchardef\zvr="7125 
\mathchardef\zvs="7126 
\mathchardef\zf="7127  
\mathchardef\zG="7000  
\mathchardef\zD="7001  
\mathchardef\zY="7002  
\mathchardef\zL="7003  
\mathchardef\zX="7004  
\mathchardef\zP="7005  
\mathchardef\zS="7006  
\mathchardef\zU="7007  
\mathchardef\zF="7008  
\mathchardef\zW="700A  
\mathchardef\zC="7009  

\newcommand{\be}{\begin{equation}}
\newcommand{\ee}{\end{equation}}

\newcommand{\bea}{\begin{eqnarray}}
\newcommand{\eea}{\end{eqnarray}}
\def\*{{\textstyle *}}
\newcommand{\R}{{\mathbb R}}

\newcommand{\s}{{\textstyle *}}

\def\ul{\underline}
\def\Sec{\sss{Sec}}

\def\Vect{\sss{Vect}}

\def\sH{{\sss H}}

\def\sT{{\sss T}}

\def\sV{{\sss V}}

\def\xi{\tx{i}}

\def\s*{{\scriptstyle *}}

\def\cO{\mathcal{O}}

\def\ul{\underline}

\newcommand{\beas}{\begin{eqnarray*}}
\newcommand{\eeas}{\end{eqnarray*}}

\title{Go Fishing with Alice: Compatible Carrollian and Poisson Structures} 
\author{Andrew James Bruce \,\orcidlink{0000-0001-8197-2263} } 
   \address{Independent Researcher, Cardiff, United Kingdom }  
   \email{andrewjamesbruce@googlemail.com}

   \date{\today}
 
\begin{document}
 \maketitle
\vspace{-20pt}
\begin{abstract}{\noindent We introduce, motivated by noncommutative Carrollian geometry, the notion of a Carrollian--Poisson bundle understood as a Carrollian manifold with a compatible Poisson structure. Compatibility is expressed as the inclusion of the Carroll distribution in the image of the Poisson anchor.  We explore the geometric consequences of this compatibility and present several simple examples of Carrollian--Poisson bundles, including the (outer) Kerr horizon, the BTZ horizon, and future null infinity. Furthermore, we establish that every Carrollian--Poisson bundle can be equipped with a Vaisman (contravariant) connection that is both  metric-compatible and clock-compatible, and thus geometric kinematics can be defined and studied.}\\
\noindent {\Small \textbf{Keywords:} Carrollian--Poisson Bundles;~Carrollian Geometry}
\end{abstract}
\medskip
\begin{flushright}
\noindent \emph{``No wise fish would go anywhere without a porpoise.''}\\
\textsc{Lewis Carroll}, \textit{Alice's Adventures in Wonderland} (1865) 
\end{flushright}
\tableofcontents
\section{Introduction}
On the smallest scales, it is expected that spacetime will deviate radically from its classical description as a pseudo-Riemannian manifold. Indeed, in 1916, Einstein argued that quantum effects must modify general relativity.  Various approaches to quantum gravity, notably string theory, loop quantum gravity/spin foams, and non-perturbative approaches based on asymptotic safety, all strongly suggest the noncommutative nature of spacetime. For an introduction to noncommutative geometry, the reader may consult Connes \cite{Connes:1994} and/or Madore \cite{Madore:1999}. For a review of quantum gravity, the reader may consult Bambi (ed.) et al. \cite{Bambi:2024}. Manifolds with both Poisson and (pseudo-)Riemannian structures have been studied as a bridge between Einsteinian gravity and the noncommutative structure expected in quantum gravity.  Loosely, the Poisson bracket is seen as the residue of the full quantum nature of spacetime, and this is especially clear when viewed through the lens of deformation quantisation.  In the literature, various compatibility conditions between the metric and Poisson bivector have been proposed; we direct the reader to Beggs \& Majid \cite{Beggs:2017}, Boucetta \cite{Boucetta:2001,Boucetta:2003,Boucetta:2004}, Hawkins \cite{Hawkins:2004}, Martínez \& Vargas \cite{Martínez:2025}, and Kuntner \&  Steinacker \cite{Kunter:2012}, for example.  As metrics and Poisson structures are very different objects, compatibility is commonly a differential condition, such as the Poisson structure being covariantly constant with respect to the Levi-Civita connection of the metric.  \par 
The original works on what is now known as \emph{Carrollian physics}, i.e.,  physics in the ultra-relativistic limit $c\rightarrow 0$, are those of Lévy-Leblond \cite{Lévy-Leblond:1965} and  Sen Gupta \cite{SenGupta:1966}.  For a review of some of the modern facets of Carrollian physics, the reader may consult Bagchi et al. \cite{Bagchi:2025}. The study of \emph{Carrollian manifolds}, so the intrinsic approach to the geometry of the ultra-relativistic limit was initiated by  Henneaux \cite{Henneaux:1979}, and further expounded upon by  Duval et al. \cite{Duval:2014a,Duval:2014b,Duval:2014c}. A Carrollian manifold is a smooth manifold equipped with a degenerate metric whose kernel is rank-$1$. The kernel is usually assumed to have a chosen generator known as the Carroll vector field. A large part of the renewed interest in such manifolds is their role in flat space holography:  the boundary theory lives on future null infinity $\mathscr{I}^+ \cong \mathbb{R} \times S^{d-2}$, which comes with a Carrollian structure, see  \cite{Bagchi:2024,Ruzziconi:2026} for example, and references therein.  Carrollian manifolds may be reformulated, under some physically reasonable conditions,  in terms of principal $\R$-bundles\footnote{Here the principal action is the additive action. Principal $\R$-bundles are equivalent to bundles of affine values (AV-bundles) and defined and applied to `gauge-independent' mechanics by Grabowska et al. \cite{Grabowska:2004}.} as highlighted by  Ciambelli et.al. \cite{Ciambelli:2019}, and it is this framework we will employ in this paper. The key advantages of taking a principal bundle approach are that one has a very clear mathematical understanding of clock forms and one can employ language familiar from gauge theory. Our standard references for principal bundle are Kolář et al. \cite[Chapter III]{Kolář:1993} and Tu \cite[Chapter 6]{Tu:2017}. \par 
In this paper, we explore a potential route to effective geometric backgrounds expected to  come from quantum Carrollian gravity (for the classical case see \cite{March:2025}).  Interestingly, there have been very few works on noncommutative Carrollian geometry, see \cite{Ballesteros:2020,Ballesteros:2020b,Ballesteros:2023,Bose:2025,Bruce:2026,Trześniewski:2024}. With noncommutative geometry in mind, we formulate the mathematical theory of Carrollian manifolds equipped with a compatible Poisson structure, which we interpret as encoding the classical limit of the underlying noncommutative geometry. The question is what compatibility condition(s) can be imposed. \par 
One potential route to `quantum/noncommutative Carrollian geometry' is to select a classical Carrollian manifold and deform it. That is, one can look for a noncommutative algebra that in some appropriate sense approximates the commutative algebra of smooth functions on a Carrollian manifold.   At a minimum, a `noncommutative Carrollian space'  is a noncommutative algebra $\mathcal{A}$ equipped with a `Carroll derivation' $\widehat{\kappa}\in\Der(\mathcal{A})$ and a metric $\widehat{g} : \Der(\mathcal{A}) \otimes_\mathcal{A} \Der(\mathcal{A}) \rightarrow \mathcal{A}$, such that $\ker(\widehat{g}) = \Span\{ \widehat{\kappa}\}$; a variation of this approach was presented in \cite{Bruce:2026}. Assuming the algebra has a consistent quasi-classical limit,  the correspondence principle states\footnote{Here $\hbar$ is a formal deformation parameter and not necessarily the reduced Planck's constant.}
$$[f_1, f_2]_* = \rmi \hbar \{f_1, f_2 \} + \cO(\hbar^2)\,.$$
Thus, to first-order, we have a Poisson structure.  We insist that $\widehat{\kappa}$ is an inner derivation, so that the `flow' in the `null direction' is defined by the algebra and is not an `external property'. This is motivated by the philosophy of general covariance in quantum gravity, where we require time as a property of the algebra.   Thus,
$$\widehat{\kappa}(f) = \kappa(f) + \cO(\hbar) = - \frac{1}{\rmi \hbar} [h, f]_* =  \{h, f\} + \cO(\hbar) = P^\#(\rmd h)(f) + \cO(\hbar)\,,$$
for some (not necessarily unique) Hamiltonian $h \in \mathcal{A}$. Here $P^\#$ is the Poisson anchor associated with the underlying classical Poisson structure.  In the other direction, by equipping a Carrollian manifold with a Poisson structure, subject to a compatibility condition, it may be possible to apply geometric quantisation and construct a series of algebras that in a strong sense approximate the smooth functions on a Carrollian manifold. In a weaker sense, Kontsevich \cite{Kontsevich:2003} has shown that any Poisson manifold admits a formal deformation, generalising earlier results of Fedosov \cite{Fedosov:1994}.   \par 
With the above comments in mind, we define a \emph{Carrollian--Poisson bundle} as a principal $\R$-bundle $\pi : M \rightarrow \Sigma$, equipped with a degenerate metric $g$, whose kernel coincides with the vertical vector fields, and a Poisson structure $P$,  such that the compatibility condition $\ker_m(g) \subset \textrm{Im}(P_m^\#)$ hold for all points $m \in M$ such that the linear map $P^\#_m : \sT_m^* M \rightarrow \sT_m M$ is of non-zero rank. We remark that the compatibility condition is trivially satisfied if the Poisson structure is symplectic, i.e., the Poisson anchor is an isomorphism.  However, most physically relevant Carrollian manifolds are three-dimensional, and so cannot admit symplectic structures.\par   
The compatibility condition ensures that  for every $u \in \ker_m(g) \subset \sT_m M$, there exists at least one $\alpha \in \sT^*_m M$ such that $P^\#(\alpha) = u$,  at points $m \in M$ at which the rank of $P^\#_m$ is non-zero. That is, away from zero points, flows along the null/fibre direction are ``locally Hamiltonian'' in this slightly generalised sense, i.e., they are generated by the Poisson structure and particular local  one-forms.  The compatibility condition, stated algebraically, is equivalent to the geometric requirement that the Carroll distribution (i.e., the null direction) must be tangent to the  non-zero dimensional leaves of the symplectic foliation.\par 
From the point of view of principal bundles, it is natural to insist that the principal $\R$-action is a Poisson map, or in other words that the Poisson structure is equivariant. Infinitesimally, this is just the condition that $\mathcal{L}_\kappa P =0$. While this condition is vital in ensuring that $\kappa$ is locally Hamiltonian in the standard sense, at least away from the zero points,  it is unimportant for the general geometry. Thus, we will not impose that the Poisson structures be equivariant; this is seen as an additional, if natural, constraint. \par 
We define a \emph{strong Carrollian--Poisson bundle} as a Carrollian--Poisson bundle equipped with a contravariant connection, also known as a Vaisman connection, that is metric-compatible and clock-compatible. We view strong Carrollian--Poisson bundles as  potential geometric models of a low-energy effective  background coming from a Carrollian quantum theory.  By using Vaisman connections, we define contravariant geodesics understood as the trajectories of test particles.  We remark that such compatible Vaisman connections are not unique  and generally, non-canonical. \par  
We stress that our approach is intrinsic in the sense that we do not insist that the Carrollian--Poisson structures are necessarily inherited from a Riemannian--Poisson manifold (or similar) as an ultra-relativistic limit nor as a null hypersurface. Alongside various mathematical statements, we present simple examples throughout the paper, and most notably, we show that the  (outer) Horizon of a Kerr black hole and (future) null infinity are naturally  Carrollian--Poisson bundles. Thus, while there are simple mathematical examples, physically relevant examples exist.  
\medskip 

\noindent{\textbf{Key Results:}}
A Carrollian--Poisson bundle is a quadruple $(M, g, \tau, P)$, where $M$ is a principal $\R$-bundle, $g$ is a degenerate metric, $\tau$  is a clock form (an Ehresmann connection) and $P$ is a Poisson structure  (see Definition \ref{def:CarrPoisBund}). Alongside other results, we establish the following. 
\begin{description}
\item[Theorem \ref{the:CarrCotLieAlg}] If the Poisson structure is regular, then there is a Lie subalgebroid $A \subset \sT^*M$ referred to as the Carrollian cotangent Lie algebroid, whose sections are one-forms such that $P^\#(\alpha) \in \ker(g)$.
\item[Proposition \ref{prop:GlobDecom}] Poisson structures on Carrollian--Poisson bundles admit a global decomposition $P = \kappa \wedge E_\mathsf{H} + P_\mathsf{H}$, where $\kappa$ is the fundamental (Carrollian) vector field, and $E_\mathsf{H}$ and $P_\mathsf{H}$ are purely horizontal as defined by the clock form. If $P_\sH=0$, we speak of minimal structures.
\item[Theorem \ref{trm:LocalHam}] If the Poisson structure is equivariant, i.e., independent of time, then the fundamental vector field $\kappa \in \ker(g)$ is locally Hamiltonian on the non-zero dimensional symplectic leaves. 
\item[Theorem \ref{the:MetClocCon}] Metric-compatible and clock-compatible Vaisman (contravariant) connections exists on all Carrollian--Poisson bundles, albeit non-canonically. 
\item[Proposition \ref{prop:IniConMin}] On strong minimal Carrollian--Poisson bundles,  contravariant geodesics that are confined to the fibres of $M$ can be formulated as an initial condition problem. 
\end{description}
\medskip

\noindent{\textbf{Organisation of the paper:}} In Subsection \ref{subsec:Rec}, we briefly recall the facets of Carrollian and Poisson geometry as needed in later subsections.  The main definitions and first examples are presented in Subsection \ref{subsec:ComCarrPois}.  In Subsection \ref{subsec:GeoCarPois}, the geometry and algebra of Carrollian--Poisson bundles is explored. Further examples of Carrollian--Poisson bundles are presented in Subsection \ref{subsec:MoreExamples}. In Subsection \ref{subsec:CompContConn}, we explore contravariant or Vaisman connections that are both metric-compatible and clock-compatible; specifically, we establish that such connections always exist.  With the existence of compatible Vaisman connections, in Subsection \ref{subsec:GeoKin} we make initial study of geometric kinematics via contravariant geodesics. We end with Section \ref{sec:ConMark} with concluding remarks.  
\medskip

\noindent{\textbf{Conventions:}} All manifolds will be real, finite-dimensional, Hausdorff, second countable and smooth.
%
%
\section{Carrollian--Poisson Bundles}
\subsection{Recollection of Carrollian and Poisson Geometry}\label{subsec:Rec}
We approach Carrollian geometry from the principal bundle perspective (see Ciambelli et al. \cite{Ciambelli:2019}). That is, spacetime is understood as a principal $\mathbb{R}$-bundle $\pi:M\rightarrow\Sigma$. While all principal $\mathbb{R}$-bundles are topologically trivial (as $\mathbb{R}$ is contractible), the choice of a specific diffeomorphism $M\cong\mathbb{R}\times\Sigma$ is equivalent to choosing a global section or a specific `time-slicing'. To remain strictly covariant and avoid a dependence on an observer's specific clock synchronisation, we treat $M$ as a principal bundle without assuming a preferred global trivialisation. Moreover, our approach will be coordinate free to further ensure covariant nature of the constructions.\par 
A Carrollian metric on $M$  is a degenerate metric $g \in \Sec(\sT^* M \otimes \sT^*M)$, such that $\ker(g) = \Vect_{\mathsf{V}}(M)$. That is, the null directions define the fibres of the principal bundle. Note, we do not, in general, require the degenerate metric to be equivariant. Typically, one includes a \emph{clock form} $\tau \in \Omega^{1}(M)$, which mathematically is a choice of Ehresmann connection (not necessarily principal nor flat). The clock form defines the decompositions 
$$\Omega^1(M) =\Omega^1_\mathsf{V}(M) \oplus \Omega^1_\mathsf{H}(M)\,, \qquad \Vect(M) = \Vect_\mathsf{V}(M) \oplus \Vect_\mathsf{H}(M)\,.$$
That is, we have a clear notion of vertical one-forms and horizontal vector fields. In particular, the clock form gives a (global) frame of $\Omega^1_\mathsf{V}(M)$, while a fundamental field $\kappa$, also known as a \emph{Carroll vector field},  defines a (global) frame of $\Vect_\mathsf{V}(M)$.  Note that as we fix a clock form, the property $\tau(\kappa)= 1$ fixes our choice of Carroll vector field. The  important defining  property of the clock form is that $\ker(\tau) = \Vect_\mathsf{H}(M)$. In this paper, we will insist that $g(X,X) >0$ for all non-zero $X \in \Vect_\mathsf{H}(M)$. In general, this condition may be relaxed, see Gibbons \cite{Gibbons:2019}. \par
If the clock form $\tau$ is closed, i.e., $\rmd \tau =0$, then we have a well-defined notion of absolute simultaneity. In particular, via the Frobenius Theorem, $\ker(\tau) = \Vect_\mathsf{H}(M)$ is integrable. This implies that $M$ is foliated by spatial slices, allowing for global synchronisation of clocks. \par
The triple $(M, g ,\tau)$, we refer to as a \emph{Carrollian bundle} to distinguish this approach from other approaches found in the literature. We will consider Carrollian bundles equipped with a Poisson structure. For an introduction to Poisson geometry, the reader may consult Crainic et al. \cite{Crainic:2021}. Our standard reference for Lie algebroids is Mackenzie \cite{Mackenzie:2005}. For details of singular foliations, the reader may examine  Laurent-Gengoux et al.  \cite{Laurent-Gengoux:2025}.  \par 
The $C^\infty(M)$-module of $q$-vector fields (multivector fields of degree/order $q$) is defied as $\mathfrak{X}^q(M):= \Sec(\wedge^q \sT M)$. The (graded) module of multivector fields  is $\mathfrak{X}^\bullet(M) := \bigoplus_{k= 0}^d \mathfrak{X}^k(M)$, where $d$ is the dimension of the manifold $M$. Under the wedge product, multivector fields for a graded graded-commutative algebra. Moreover, multivector fields come equipped with the Schouten--Nijenhuis  bracket
 $$\SN{-,-} : \mathfrak{X}^p(M) \times \mathfrak{X}^q(M) \rightarrow \mathfrak{X}^{p+q-1}(M)\,,$$ 
which satisfies
\begin{align}\label{eqn:SNProps}
& \SN{X,Y} =  - (-1)^{( |X| - 1)(|Y|-1)}\, \SN{Y,X}\,,\\ \nonumber
&\SN{X,Y\wedge Z} =  \SN{X,Y}\wedge Z + (-1)^{(|X| + 1)|Y|  } \, Y \wedge \SN{X,Z}\,,\\ \nonumber 
&\SN{X, \SN{Y,Z}} =  \SN{\SN{X,Y}, Z} + (-1)^{(|X|+1)( |Y| +1) } \,  \SN{Y, \SN{X,Z}}\,,
\end{align}
where we denote the degree of a multivector field as $|X|$, i.e., if $X \in \mathfrak{X}^q(M)$, then $|X| = q$.  Note that for functions $f_1, f_2 \in C^\infty(M)$, considered as $0$-vector fields,  $\SN{f_1, f_2} =0$, and for vector fields ($1$-vector field) $\SN{X,Y} = [X,Y]$, where $[-,-]$ is the standard Lie bracket of vector fields.  The reader may consult \cite[Chapter 2]{Crainic:2021} for further details of the Schouten--Nijenhuis bracket. \par
 Recall that a Poisson structure is a bivector $P \in \mathfrak{X}^2(M)$ that satisfies $\SN{P,P} =0$. The pair $(M, P)$ is a \emph{Poisson manifold}.  The cotangent bundle $\sT^* M$  of a Poisson manifold is canonically a Lie algebroid. For completeness, we remind the reader of the definition of a Lie algebroid.
 \begin{definition}\label{def:LieAlg}
A \emph{Lie algebroid} is a triple $(A, [-,-], \rho)$, where $\pi : A \rightarrow M$ is a vector bundle, $[-,-]$ is a Lie bracket on the vector space $\Sec(A)$, referred to as the \emph{Lie algebroid bracket}, and  a linear map $\rho: \Sec(A) \rightarrow \Vect(M)$, referred to as the \emph{anchor}, that satisfies the Leibniz rule
$$[u,f\,v] = \rho(u)f \, v  +  f \, [u,v]\, ,$$
with $u$ and $v \in \Sec(A)$ and $f \in C^\infty(M)$. 
\end{definition}
It should be noted that one can deduce that the anchor is in fact a homomorphism of Lie algebras, i.e., $\rho[u,v] = [\rho(u), \rho(v)]$. For the case at hand, the anchor is defined as
\begin{align}
 & P^\#  : \Omega^1(M) \longrightarrow \Vect(M)\\ \nonumber 
 &  P^\#(\alpha) := P(\alpha, -)\,,
\end{align}
for all $\alpha \in \Omega^1(M)$. The distribution $\mathrm{Im}(P^\#) \subset \Vect(M)$ is referred to as the \emph{characteristic distribution}.  Note that the characteristic distribution is an integrable Stefan--Sussmann distribution. Conventionally, we set $\beta(P^\#(\alpha)) := P(\alpha, \beta)$, and to be consistent with this, we define $i_X \alpha  = - \alpha(X)$, which is a common convention in symplectic/Poisson geometry. \par 
The Lie algebroid bracket, referred to as the \emph{Koszul bracket}, is defined as 
\begin{equation}
[\alpha , \beta] := \mathcal{L}_{P^\#(\alpha)} \beta - \mathcal{L}_{P^\#(\beta)} \alpha - \rmd \big( P(\alpha, \beta)\big)\,, 
\end{equation} 
for all $\alpha$ and $\beta \in \Omega^1(M)$. Note, we have the Leibniz rule 
$$[\alpha , f \, \beta] = \mathcal{L}_{P^\#( \alpha)} f \, \beta + f \, [\alpha , \beta]\,,$$
for all $\alpha , \beta \in \Omega^1(M)$ and $f \in C^\infty(M)$. \par 
Note that if a one-form $\alpha \in \Omega^1(M)$ is  de Rham closed, i.e., $\rmd \alpha =0$, then $P^\#(\alpha)$ is locally Hamiltonian and thus a Poisson vector field;  that is, $\mathcal{L}_{P^\#(\alpha)}P =0$.  The Poisson bracket on $C^\infty(M)$ is defined as $\{ f_1, f_2\} := P(\rmd f_1, \rmd f_2)$.  We remark that the Poisson bracket will not feature heavily in the remainder of this paper. \par 
We say that a point $m \in M$ is a \emph{regular point} if the rank of $P$ at the point, i.e., the rank of the linear map
$$P^\#_m : \sT^*_m M \rightarrow \sT_m M\,,$$
is constant on a neighbourhood $U_m$ of $m$.  Otherwise, the point is called a \emph{singular point}.  A Poisson structure $P$ is said to be a \emph{regular Poisson structure} if the rank is constant, i.e, the same value at all points. Note that as a Poisson structure is skewsymmetric, the rank at any point is an even integer.\par
The \emph{symplectic foliation} is the foliation associated with the characteristic distribution (see \cite[Chapter 4]{Crainic:2021}). Each leaf of a symplectic foliation is a symplectic manifold, and so must be even-dimensional. If the Poisson structure is regular, then all the symplectic leaves have the same dimension, i.e., the foliation is regular.
%
%
\subsection{Carrollian--Poisson Bundles and First Examples}\label{subsec:ComCarrPois}
A Carrollian--Poisson bundle, as we will define them, is a Carrollian bundle equipped with a Poisson structure. The question is, ``How do these structures interact?''  Given that our main perspective is that the resulting geometry should represent the quasi-classical limit of Carrollian noncommutative geometry, we make the following definition. 
\begin{definition}\label{def:CarrPoisBund}
A \emph{Carrollian--Poisson bundle} is a quadruple $(M, g, \tau,  P)$, where $(M,g, \tau)$ is a Carrollian bundle and  $(M, P)$  is a Poisson manifold, subject to the compatibility condition 
$$\ker_m(g)\subset \textrm{Im}(P_m^\#)\,,$$
at all points $m \in M$ where the rank of $P_m^\# : \sT_m^* M \rightarrow \sT_m M$ is non-zero. \par 
If the principal $\R$-action is Poisson, i.e., $P$ is equivariant, then we speak of a \emph{Poisson-stationary Carollian--Poisson bundle}. If the Poisson structure is symplectic, then we speak of a  \emph{Carrollian--symplectic bundle}. If the Poisson structure is regular, then we speak of a \emph{regular Carrollian--Poisson bundle}.
\end{definition}
\begin{example}\label{exa:Trivial}
Any Carrollian bundle $(M,g, \tau)$ can be equipped with the trivial (zero) Poisson structure $P=0$. As there are no points $m \in M$ where $P^\#$ is not of rank zero, the compatibility condition is vacuously satisfied. 
\end{example}
\begin{example}\label{exa:Symp}
Let $(M, g, \tau)$ be a Carrollian bundle such that $(M, \omega)$ is a symplectic manifold. Let is denote the associated Poisson structure as $P := \omega^{-1}$. As, for every point $m \in M$,  we  have an  isomorphism $P_m^\# :  \sT_m^* M \Diffeo \sT_m M$, the condition $\ker_m(g) \subset \mathrm{Im}(P_m^\#)$ is trivially satisfied for all points. 
\end{example}
Examples \ref{exa:Trivial} and \ref{exa:Symp} represent the the two edge cases for Carrollian--Poisson bundles. For the trivial case, all one-forms are in the kernel of the Poisson anchor. While for the symplectic case, no one-forms are in the kernel of the Poisson anchor. 
\begin{example}\label{exa:CarrPlane} 
Consider the Carrollian plane $(\R^2, g)$, where $g$ is the pullback of the Euclidean metric $h$ on $\R$, i.e., $g = \pi^* h$.  Let $\Omega$  be the canonical volume form on $\R$ associated with $h$. Then we have the unique (nowhere vanishing) vector field $v \in \Vect(\R)$ defined by $h(v,v) =1$ and $\Omega(v) >0$. Then for any choice of clock form $\tau$, we define $P = \psi\, \kappa \wedge v^\mathsf{H}$ for any $\psi \in C^\infty(\R^2) $ which may have zeros.  As degree three multivector fields on $\R^2$ vanish, $P$ is a Poisson structure.  Choosing the frame for  $\Omega^1_\mathsf{H}(\R^2)$ dual to $v^\mathsf{H}$ using $g$, which we denote as $\zx$, we have $P^\#(\tau) = \psi \, v^\mathsf{H}$ and $P^\#(\zx) = - \psi \, \zx(v^\mathsf{H})\, \kappa = - \psi\, \kappa$.  Thus, $\mathrm{Im}(P_m^\#) = \Span \{\kappa_m, v_m^\mathsf{H} \}$ at points where $\psi_m \neq 0$, demonstrating the compatibility condition $\ker_m(g) \subset \mathrm{Im}(P_m^\#)$.  
\end{example}
\begin{example}\label{exa:R3}
Consider the Carrollian bundle  $(\R\times \R^2, g,\tau)$ where we have $g = \pi^* h$, with $h$ is an arbitrary metric on $\R^2$. We will choose a closed equivariant clock form, so that the horizontal distribution is integrable. We then equip $\Vect(\R^2)$ with an orthonormal frame $\{v,w \}$, i.e., $h(v,v)= h(w,w) =1$ and $h(v,w)=0$. Note all such frames are isomorphic.  Then using the clock form, we define $P = \kappa \wedge (v^\mathsf{H} + w^\mathsf{H})$. As the clock is closed and equivariant, so equivalent to a flat principal connection (which always exists in this case), we have $[\kappa, v^\mathsf{H}] = [\kappa, w^\mathsf{H}] =0$, $P$ is Poisson.  Note this structure is not symplectic. We pick a frame of $\Omega^1_\mathsf{H}(\R^3)$ dual to the orthonormal frame, we denote this as $\{\zx, \eta \}$. Then $P^\#(\tau) = v^\mathsf{H} + w^\mathsf{H}$, $P^\#(\zx) = - \kappa$ and $P^\#(\eta) = - \kappa$. Thus, $\mathrm{Im}(P^\#) =  \Span\{\kappa, v^\mathsf{H}+ w^\mathsf{H}\}$, demonstrating the compatibility condition  $\ker_m(g) \subset \mathrm{Im}(P_m^\#)$ for all points $m \in M$. Note $\alpha = - (a \, \zx + (1-a)\, \eta)$ with $a \in [0,1]$, is such that $P^\#(\alpha) = \kappa$. Choosing a different orthonormal bases of the vector fields on $\R^2$ will yield  isomorphic Poisson structures, see the next definition.
\end{example}
\begin{definition}
Let $(M, g , \tau, P)$ and $(M', g' , \tau' , P')$ be Carrollian--Poisson bundles.  An \emph{isomorphism of Carrollian--Poisson bundles} is a pair of diffeomorphisms  $\Phi : M \Diffeo M'$ and $\phi : \Sigma \Diffeo \Sigma'$ such that 
\begin{align*}
&\Phi(m \triangleleft r ) = \Phi(m)\triangleleft r\,, && \pi' \circ \Phi = \phi \circ \pi \,,&&  &\\
&g = \Phi^*g'\, && \tau = \Phi^* \tau' \, , && \Phi_* P = P'\,,&
\end{align*}
for all $m \in M$ and $r\in \R$, i.e., an isomorphism of  principal $\R$-bundles that respects the Carrollian and Poisson structures.  The resulting category of \emph{Carrollian--Poisson bundles} we denote as $\catname{CarPois}$.
\end{definition}
\begin{remark}
As all the morphisms in $\catname{CarPois}$ are isomorphisms, the category is a groupoid.  In general relativity, the principle of general covariance states that the physical content of the theory is invariant under diffeomorphisms. Similarly, we interpret isomorphic Carrollian--Poisson bundles as  ``physically equivalent''.
\end{remark}
We will consider more involved examples in Subsection \ref{subsec:MoreExamples}.
%
%
\subsection{The Geometry of Carrollian--Poisson Bundles}\label{subsec:GeoCarPois}
We will proceed to explore the geometry of Carrollian--Poisson bundles. Several fundamental statements from the general theory are established. 
\begin{observation}\label{obs:PHashSur}
Let $(M, g, \tau, P)$ be a Carrollian--Poisson bundle, then  the compatibility condition $\ker_m(g) \subset \textrm{Im}(P_m^\#)$ for all points $m \in M$ for which the rank of $P_m^\#$ is non-zero implies that 
$$P_m^\# :  \sT_m^* M \rightarrow \sT_m M\,,$$
is a surjection onto $\mathsf{V}_m M$. Moreover, if the rank of $P$ never drops to zero, then 
$$P^\# : \Omega^1(M)\longrightarrow \Vect(M)\,, $$
is a surjection onto the submodule  $\Vect_\mathsf{V}(M)$.
\end{observation}
As we fix a clock form, we will speak of \ul{the} fundamental vector field which is defined by $\tau(\kappa) = 1$.  In other words, we have chosen a nowhere vanishing vector field that spans the kernel of the metric via the clock form.
\begin{proposition}\label{prop:HorNoVan}
Let $(M, g, \tau, P)$ be a Carrollian--Poisson bundle, then the vector field $P^\#(\tau)\in \Vect(M)$ is horizontal.
\end{proposition}
\begin{proof}
Observe  that  $\tau(P^\#(\tau)) = P(\tau, \tau)=0$ via skewsymmetry. Thus, $P^\#(\tau)$ is horizontal as it lies in the kernel of the clock form. 
\end{proof}
Let us define the submodule of one-forms,  
$$\mathcal{A}_\mathsf{C}(M):= \big \{ \eta \in \Omega_\mathsf{H}^1(M)~|~ P^\#(\eta) \in \ker(g) \big\}\,,$$
which we refer to as \emph{Carrollian one-forms}. Note that Observation \ref{obs:PHashSur} ensures that this module is non-trivial.  It is important to note that, in general,  not all horizontal one-forms are Carrollian one-forms.
\begin{proposition}\label{prop:GlobDecom}
Let $(M, g, \tau, P)$ be a Carrollian--Poisson bundle. Then, $P$ has a unique global decomposition
$$P = \kappa \wedge E_\mathsf{H} + P_\mathsf{H}\,,$$
where $E_\mathsf{H} \in  \Vect_\mathsf{H}(M)$, and  $P_\mathsf{H} \in \wedge^2 \Vect_\mathsf{H}(M)$ is a horizontal bivector field, such that $P^\#(\tau) = E_\mathsf{H}$ and $P_\mathsf{H}^\#(\eta) =0$  for any $\eta\in \mathcal{A}_\mathsf{C}(M)$.
\end{proposition}
\begin{proof}Using the clock form, we have the decomposition
$$\mathfrak{X}^2(M) = \Vect_\mathsf{V}(M) \wedge \Vect_\mathsf{H}(M) \oplus  \wedge^2 \Vect_\mathsf{H}(M)\,.$$
We have the decomposition $P = \kappa \wedge E_\mathsf{H} + P_\mathsf{H}$ as the fundamental vector field spans the vertical vector fields.  Clearly, $P^\#(\tau) = E_\mathsf{H}$ as claimed.   Next consider any  $\eta \in \mathcal{A}_\mathsf{C}(M)$, then directly
$$P^\#(\eta) = - \eta(E_\mathsf{H})\, \kappa + P_\mathsf{H}^\#(\eta)\,. $$
As this must be vertical, we deduce that $P_\mathsf{H}^\#(\eta)=0$.\par 
To establish uniqueness suppose $P = \kappa \wedge E_\mathsf{H} + P_\mathsf{H} = \kappa \wedge E'_\mathsf{H} + P'_\mathsf{H}$. Then $P^\#(\tau) = E'_\mathsf{H} = E_\mathsf{H}$, establishing that the horizontal vector field is unique. This directly implies $P_\mathsf{H} = P'_\mathsf{H}$.  
\end{proof}
\begin{corollary}
Let $(M, g, \tau , P)$ be a Carrollian--Poisson bundle. Then $\ker(P_\mathsf{H}^\#)$ is non-trivial.
\end{corollary} 
\begin{remark}
The global decomposition given in Propositions \ref{prop:GlobDecom} formally resembles a Jacobi structure (see \cite{Bruce:2017} and references therein). However, the principal bundle structure here is $G = (\R, +)$ rather than $\mathrm{GL}(1, \R) = (\R^\times, \cdot)$.  Moreover, the clock form plays a vital role in the decomposition. 
\end{remark}
\begin{small}
\noindent \textbf{Aside:} Although we define Carrollian--Poisson bundles to have a fixed clock form $\tau$, it is illustrative to relax this and consider gauge transformations.  We define a gauge transformation of the clock form as 
$$\tau \mapsto \tau' := \tau + \alpha_\mathsf{H}\,,$$
with $\alpha_\mathsf{H} \in \Omega^1_\mathsf{H}(M)$ and $\rmd \alpha_\mathsf{H} =0$, i.e.,  we shift the clock by a closed horizontal one form.  Note that this shift does not effect the Carroll vector field $\kappa$. Directly from the decomposition of Poisson structure and noting $P = \kappa \wedge E_\mathsf{H} + P_\mathsf{H} = \kappa \wedge E'_\mathsf{H} + P'_\mathsf{H}$ as the Poisson structure is a globally defined object irrespective of the decomposition, we have 
\begin{align*}
 E'_\mathsf{H} &= E_\mathsf{H} - \alpha_\mathsf{H}(E_\mathsf{H}) \, \kappa + P_\mathsf{H}^\#(\alpha_\mathsf{H})\,,\\
 P'_\mathsf{H} & = P_\mathsf{H} - \kappa \wedge P_\mathsf{H}^\#(\alpha_\mathsf{H})\,.
\end{align*} 
It must be stressed that $E'_\mathsf{H}$ and $P'_\mathsf{H}$ are horizontal with respect to $\tau'$; this can be directly checked. 
\end{small}
\begin{proposition}\label{prop:MinLowDim}
Let $(M, g, \tau, P)$ be a Carrollian--Poisson bundle such that $1\leq n \leq 2$, then $P_H = 0$ and for any point $m \in M$, $P_m^\# : \sT^*_m M \rightarrow \sT_m M$ is either of rank $0$  or $2$.
\end{proposition}
\begin{proof}
Using the fundamental vector field $\kappa \in \Vect_\mathsf{V}(M)$ we have the decomposition
$$P = \kappa \wedge E_\mathsf{H} + P_\mathsf{H}\,.$$ 
\begin{description}
\item[$n = 1$] As $\Vect_\mathsf{H}(M)$ is one-dimensional $\wedge^2 \Vect_\mathsf{H}(M)$ is trivial, and thus $P_\mathsf{H}=0$.
\item[$n = 2$] Here $\Vect_\mathsf{H}(M)$ is two-dimensional. Any skewsymmetric linear map  must have even rank. Thus, $P^\#_\mathsf{H}$ must be of rank $0$ or $2$ (at every $m \in M$). As any (non-zero) $\eta \in \mathcal{A}_\mathsf{C}(M)$ is in the kernel of $P^\#_\mathsf{H}$, it cannot be of full rank.  Thus, the rank must be zero and hence $P_\mathsf{H} =0$.
\end{description}  
Since $P_\mathsf{H}=0$ in both cases, the Poisson structure is (globally) $P = \kappa \wedge E_\mathsf{H}$. As $\kappa$ is vertical and  $E_\mathsf{H}$ is horizontal, they belong to complementary distributions.  At points $m \in M$ where $E_{\mathsf{H}, m} \neq 0$, $\kappa_m$ and $E_{\mathsf{H}, m}$ are linearly independent and thus, at such points $P^\#_m$ is of rank $2$.  At points where $E_\mathsf{H}$ vanishes, then the rank drops to zero.
\end{proof}
\begin{definition}
A Carrollian--Poisson bundle $(M, g , \tau, P)$ us said to be a \emph{minimal Carrollian--Poisson bundle}.  A minimal Carrollian--Poisson bundle is said to be an \emph{almost regular minimal Carrollian--Poisson bundle} if the set of points $M_0 := \{m \in M ~~|~~ P_m \neq 0\}$ is dense. 
\end{definition}
Note that for minimal Poisson structures $P^\#(\Omega^1_\mathsf{H}(M))\subseteq \ker(g)$, implying $\mathcal{A}_\mathsf{C}(M) \cong \Omega^1_\mathsf{H}(M)$.
\begin{proposition}
Let $(M, g, \tau , P)$ be a minimal Carrollian--Poisson bundle, then the distribution $\mathcal{D} := \Span\{ \kappa, E_\mathsf{H}\}$ is integrable.
\end{proposition}
\begin{proof}
Using the properties of the Schouten--Nijenhuis bracket (see \eqref{eqn:SNProps}) we observe that 
$$\SN{P,P} = - 2 \, \kappa \wedge E_\mathsf{H} \wedge [\kappa, E_\mathsf{H}] =0\,.$$
For this to be true, it must be the case that $[\kappa, E_\mathsf{H}] = f_1 \, \kappa + f_2 \, E_\mathsf{H}$ for some (possibly zero) functions $f_1, f_2 \in C^\infty(M)$. Thus,   $\mathcal{D}$ is closed under the Lie bracket and therefore integrable. 
\end{proof}
\emph{Casimir functions} are defined as standard, i.e.,
\begin{equation}
\mathsf{Cas}(M) :=  \{f \in C^\infty(M)~~|~~P^\#(\rmd f) =0 \}\,.
\end{equation} 
For a non-trivial minimal Poisson structure we have $P^\#(\rmd f) = \kappa(f)\, E_\mathsf{H} - E_\mathsf{H}(f)\, \kappa =0$. We thus require $\kappa(f) = 0$ and $E_\mathsf{H}(f) =0$.
\begin{observation}
Let $(M, g, \tau , P)$ be a minimal Carrollian--Poisson bundle with $P\neq 0$.  Then the Casimir functions are constant on the leaves of the foliation associated with the integrable distribution $\mathcal{D} = \Span \{ \kappa, E_\mathsf{H} \}$.
\end{observation}
\begin{proposition}
Let $(M, g, \tau, P)$ be a minimal Carrollian--Poisson bundle such that $g(E_\mathsf{H}, E_\mathsf{H})$ is a Casimir function, i.e., it is constant on the characteristic distribution $\mathrm{Im}(P^\#) = \Span \{\kappa, E_\mathsf{H} \}$. If the fundamental (Carroll) vector field $\kappa$ is Killing, i.e., $\mathcal{L}_\kappa g =0$, then
\begin{enumerate}[i)]
\item $\mathcal{L}_\kappa E_\mathsf{H} \in \Vect_\mathsf{V}(M)$;\label{propi}
\item $\mathcal{L}_\kappa P =0$; \label{propii}
\item $\mathcal{L}_{X_\mathsf{V}}g =0$; and  \label{propiii}
\item $\mathcal{L}_{X_\mathsf{V}}P =  -\kappa(f)\, P$, \label{propiv}
\end{enumerate}
for all $X_\mathsf{V} = f \, \kappa \in \Vect_\mathsf{V}(M)$. 
\end{proposition}
\begin{proof}
The Killing condition can be written as
$$\kappa \big(g(X,Y) \big) = g([\kappa, X], Y) + g(X, [\kappa, Y])\,.$$
Then setting $X=Y = E_\mathsf{H}$, and $\kappa \big(g(E_\mathsf{H},E_\mathsf{H}) \big)=0$ via the assumptions of the lemma, we obtain  $g([\kappa, E_\mathsf{H}], E_\mathsf{H})=0$.  As the kernel of $g$ is precisely $\Vect_\mathsf{V}(M)$, we have $\mathcal{L}_\kappa E_\mathsf{H} = [\kappa, E_\mathsf{H}] \in \Vect_\mathsf{V}(M)$. Thus, \ref{propi}) is established. Directly we have $\mathcal{L}_\kappa P = \mathcal{L}_\kappa \kappa \wedge E_\mathsf{H} + \kappa \wedge \mathcal{L}_\kappa E_\mathsf{H} =0$ which establishes \ref{propii}.   \par 
To establish \ref{propiii}), a direct calculation is required. Explicitly, from the definition of the Lie derivative of the metric and using $\kappa \in \ker(g)$
\begin{align*}
\big(\mathcal{L}_{X_\mathsf{V}}g\big)(Y,Z) &= f \, \kappa\big( g(Y,Z)\big) - g([f\, \kappa,Y],Z) - g (Y, [f \,\kappa, Z])\\
&= f \, \kappa\big( g(Y,Z)\big) - g(f [ \kappa,Y],Z) - g (Y, f [\kappa, Z])  \\
&= f \, \big( \mathcal{L}_\kappa g\big )(Y,Z) =0\, .
\end{align*}
Similarly, 
\begin{align*}
\mathcal{L}_{X_\mathsf{V}}P &= [f \, \kappa, \kappa] \wedge E_\mathsf{H} + \kappa \wedge [f \, \kappa, E_\mathsf{H}]\\
&= - \kappa(f)\, \kappa \wedge E_\mathsf{H} + f \, \kappa \wedge \mathcal{L}_\kappa E_\mathsf{H}\,.
\end{align*}
We then observe that \ref{propi}) implies $\mathcal{L}_{X_\mathsf{V}}P = - \kappa(f)\, \kappa \wedge E_\mathsf{H} = - \kappa(f)\, P$, and so \ref{propiv}) is established. 
\end{proof}
Returning to Carrollian one-forms, we have the following algebraic result.
\begin{lemma}\label{lem:LRPair}
Let $(M, g, \tau, P)$ be a Carrollian--Poisson bundle and let  $\mathcal{A}_\mathsf{C}(M)$ be the $C^\infty(M)$-module of Carrollian one-forms. Then $(\mathcal{A}_\mathsf{C}(M), [-,-], P^\#)$  is a Lie--Rinehart pair.
\end{lemma} 
\begin{proof}  We need to establish that the Koszul bracket closes on Carrollian one-forms. Let $\eta, \theta \in \mathcal{A}_\mathsf{C}(M)$ be arbitrary.
\begin{enumerate}[i)]
\item As the anchor map is a Lie algebra homomorphism, we observe that $P^\#[\eta, \theta] = [P^\#(\eta), P^\#(\theta)] \in \Vect_\mathsf{V}(M)$ as the Lie bracket closes on vertical vector fields. Thus, the image of the bracket under the anchor is a vertical vector field. 
\item To establish that $[\eta, \theta]$ is horizontal, note $P(\eta, \theta)= \theta(P^\#(\eta))=0$ as $\theta$ is horizontal and $P^\#(\eta)$ is a vertical vector field.  Then the Koszul bracket simplifies to
$$[\eta, \theta] = \mathcal{L}_{P^\#(\eta)}\theta - \mathcal{L}_{P^\#(\theta)}\eta\,.$$
Observe that 
$$\mathcal{L}_{P^\#(\eta)}(\theta(\kappa)) = (\mathcal{L}_{P^\#(\eta)}\theta)( \kappa )+ \theta([P^\#(\eta), \kappa])=0 \,.$$
As the Lie bracket closes on vertical vector fields and $\theta$ is horizontal, we have that $(\mathcal{L}_{P^\#(\eta)}\theta)(\kappa) =0$. Thus, we see that $[\eta, \theta](\kappa) =0$, and we deduce that $[\theta, \eta] \in \Omega^1_\mathsf{H}(M)$.  
\end{enumerate}
Putting the two results together, we conclude that the Koszul bracket closes on Carrollian one-forms.
\end{proof}
The Lie bracket on $\mathcal{A}_\mathsf{C}(M)$ reduces to the simple form $[\eta , \theta] = \mathcal{L}_{P^\#(\eta)} \theta - \mathcal{L}_{P^\#(\theta)} \eta$. This bracket is the difference of the vertical flow generated by the one-forms via the anchor acting on the other. More explicitly, using the definition of the Lie derivative, it can be shown that 
\begin{subequations}
\begin{align}
& P^\#(\eta) = - \eta(E_\mathsf{H})\, \kappa\,,\\
&[\eta, \theta] = - \eta(E_\mathsf{H}) i_\kappa (\rmd \theta) + \theta(E_\mathsf{H}) i_\kappa (\rmd \eta)\,.
\end{align}
\end{subequations}   
Note that $P_\mathsf{H}$ plays no role in the anchor or bracket.
\begin{theorem}\label{the:CarrCotLieAlg}
Let $(M, g, \tau, P)$ be a Carrollian--Poisson bundle.  The $C^\infty(M)$-module of Carrollian one-forms  $\mathcal{A}_\mathsf{C}(M)$ can, up to isomorphism,  be identified with the $C^\infty(M)$-module of (global) sections  (via the Serre--Swan theorem) of a  Lie algebroid $A \subset \sT^*M$ if and only if $P$ is a regular Poisson structure, i.e., $(M, g, \tau, P)$ is a regular Carrollian--Poisson bundle.
\end{theorem}
\begin{proof}
If $\mathcal{A}_\mathsf{C}(M)$ are the sections of a vector bundle, it must be the case that the fibres 
$$A_m := \{ \eta_m \in \mathsf{H}^*_m M ~|~ P_m^\#(\alpha_m)\in \mathsf{V}_m M \}\,,$$
have the same (finite) dimension for all points $m \in M$. \par 
If $P_m^\#$ is the zero map, implying that $P$ is the trivial Poisson structure as the fibres must have the same dimension, then $A\Diffeo  \mathsf{H}^* M$.  The Lie algebroid structure follows from Lemma \ref{lem:LRPair}, and in this case, we have the trivial Lie algebroid structure, i.e., the bracket and anchor are zero maps. \par 
Assuming that the Poisson structure is non-singular, i.e, the rank does not drop to zero, then directly from the Rank-Nullity Theorem, we have
$$\dim\big(A_m\big) = \dim\big(\ker(P_m^\#) \big) + \dim \big(\ker(g_m)\cap \mathrm{Im}(P_m^\#)   \big)\,.$$ 
The compatibility condition $\ker_m(g) \subset \mathrm{Im}(P_m^\#)$ implies $\big(\ker(g_m)\cap \mathrm{Im}(P_m^\#)\big) = \ker(g_m)$, which is one-dimensional for all points $m \in M$.  Thus, for $\dim(A_m)$ to be constant for all points, it must be the case that $\dim \big(\ker(P_m^\#) \big)$ is constant. This is equivalent to the definition of a regular Poisson structure. The Lie algebroid structure follows from Lemma \ref{lem:LRPair}, i.e., the sections come with the structure of a Lie--Rinehart pair.
\end{proof}
\begin{definition}\label{def:NullBund}
Let $(M, g, \tau, P)$ be a regular Carrollian--Poisson bundle. The Lie algebroid $(A, [-,-], P^\#)$, such that $\Sec(A) \cong \mathcal{A}_\mathsf{C}(M)$ we refer to as the \emph{Carrollian cotangent Lie algebroid}.
\end{definition}
Note that the Carrollian cotangent Lie algebroid is only defined for regular Carrollian--Poisson bundles. Without the regularity, we only have the algebraic structure of Carrollian one-forms. Using the clock form, we have  $\sT^*M \Diffeo \mathsf{V}^*M \oplus \mathsf{H}^*M$,  we see that $A$ is a vector subbundle of $\mathsf{H}^*M$ when the rank of $P$ is non-zero.  We will denote the vector bundle structure as $\pi_A : A \rightarrow M$. Elements of $A$ we interpret as Carrollian momenta. Observe that, via the rank of a Poisson structure being even, the range of the rank of $A$  is 
\begin{description}
\item[$\mathbf{n}$ odd] $~\Rank(A) \in \{ n, n-2, n-4, \cdots , 1\}$,
\item[$\mathbf{n}$ even] $\Rank(A) \in \{ n, n-2, n-4, \cdots , 2\}$.
\end{description}
\begin{observation}
Let $(M, g, \tau, P)$ be a minimal regular Carrollian--Poisson bundle with $P$ non-trivial.  As $P$ is of constant rank $2$, then $\Rank(A) =n$,  and  $A \simeq \mathsf{H}^*M$.
\end{observation}
\begin{lemma}\label{lem:SymFol}
Let $(M, g, \tau, P)$ be a Carrollian--Poisson bundle.  Let $m \in M$ be a point such that the rank of $P$ at $m$ is non-zero, then the fibre $\mathcal{F}_{\pi(m)} \Diffeo \R$ of the principal $\R$-bundle  $\pi : M \rightarrow \Sigma$  is contained in the symplectic leaf $S$ containing the point $m$.
\end{lemma}
\begin{proof}
Let $S$ be the symplectic leaf containing the point $m\in M$ as specified by the lemma. Note that a point belongs to a single leaf, i.e., we have a partition of the manifold  By definition the  tangent space at $m \in S$ is exactly $\textrm{Im}(P^\#_m)$. As $S$ is a smooth manifold, it has a well-defined dimension. Therefore, the rank of $P$ is constant across all the points of $S$.  As the rank of $P$ at $m$ is non-zero, the rank at all $m' \in S$ is non-zero.  Thus, the compatibility condition 
$$\ker_{m'}(g) \subset \textrm{Im}(P^\#_{m'}) = \sT_{m'} S\,,$$
holds for all $m' \in S$.  Then as $\kappa_{m'} \in \sT_{m'} S$ for all $m' \in S$, the restriction $\kappa|_S \in \Vect(S)$ is well-defined.  By the Picard--Lindelöf theorem, if a vector field is everywhere tangent to a submanifold, any integral curve that starts inside that submanifold can never leave it. Since the fibre $\mathcal{F}_{\pi(m)}$ is  the maximal integral curve of $\kappa|_S$ passing through $m$, the fibre is contained within the symplectic leaf $S$.
\end{proof}
\begin{lemma}\label{lem:SubRBund}
Let $(M, g, \tau, P)$ be a Poisson-stationary Carrollian–-Poisson bundle. If $S$ is a symplectic leaf of $M$ with non-zero dimension, then $S$ is a principal $\mathbb{R}$-bundle over a symplectic leaf of the base manifold $\Sigma$.
\end{lemma}
\begin{proof}
As the dimension of $S$ is non-zero, Lemma \ref{lem:SymFol} guarantees that $\kappa|_S$ is tangent to $S$. Thus, the flow of $\kappa$ preserves the leaves, meaning that $S$ is invariant under the principal $\R$-action. We then define $\Sigma_S :=  \pi(S)$. Because $S$ is invariant under the principal action, we have $S = \pi^{-1}(\Sigma_S)$.  \par 
It remains to establish that $\Sigma_S$ is a smooth manifold, this requires the Poisson-stationary condition. As $P$ is taken to be equivariant, i.e., the principal $\R$-action acts by Poisson map, $P$ is projectable.  That is, $\pi_* P := P_\Sigma$ is a Poisson structure on $\Sigma$, and furthermore $\pi : (M, P) \rightarrow (\Sigma, P_\Sigma)$ is a Poisson submersion.  As $\ker(\pi_{*,m}) = \Span\{ \kappa_m\}$, where $\pi_* : \sT_m M \rightarrow \sT_{\pi(m)}\Sigma $, $\pi_*$ restricts to a surjective linear map
$$T_m S = \textrm{Im}(P^\#_m) \rightarrow \textrm{Im}(P^\#_{\Sigma, \pi(m)})\,,$$
for all points $m \in S$. Surjectivity demonstrates that $\Sigma_S$ is a symplectic leaf of the reduced Poisson manifold $(\Sigma, P_\Sigma)$. Consequently, $\Sigma_S$ is a smooth submanifold of $\Sigma$. Because $\Sigma_S$ is a smooth submanifold of the base manifold, its preimage under the principal bundle projection $\pi$ naturally inherits the structure of a principal bundle. Since $S = \pi^{-1}(\Sigma_S)$, we conclude that $S$ is a principal $\mathbb{R}$-bundle over $\Sigma_S$.
\end{proof}
\begin{proposition}\label{prop:CarSymBun}
Let $(M, g , \tau, P)$ be a Poisson-stationary Carrollian--Poisson bundle and let $S \subset M$ be a symplectic leaf of non-zero dimension. Then $S$, equipped with the restricted structures, is a Carrollian--symplectic bundle. 
\end{proposition}
\begin{proof}
Let $\iota : S \hookrightarrow M$ be the inclusion map of the symplectic leaf. The metric is defined as $g_S := \iota^* g$, and the  clock form is $\tau_S := \iota^* \tau$. Lemma \ref{lem:SubRBund} established that the restriction of $\pi$ to $S$ mapping $S \rightarrow \Sigma_S$ is a principal $\mathbb{R}$-bundle. Furthermore, by the foundational properties of symplectic foliations, the restricted structure $P|_S$ is symplectic. To establish that $S$ is a Carrollian--symplectic bundle, we must demonstrate that the kernel of $g_S$ equals the vertical vector fields on $S$, that $\tau_S$ is a valid clock form, and that the compatibility condition holds.\par 
let $s \in S$ be an arbitrary point of the symplectic leaf.  As the rank of $P$ on $S$ is non-zero,  Lemma \ref{lem:SymFol}  implies that $\mathsf{V}_s M \subset \sT_s M$, and in particular $\mathsf{V}_s S = \mathsf{V}_s M$.
\begin{enumerate}
\item \label{pro:1} Let $s \in S$ be an arbitrary point on the symplectic leaf. As the $\R$ fibre at $s \in M$ lies in the symplectic leaf $S$, we have $\mathsf{V}_s S = \mathsf{V}_s M $. Suppose $u \in \mathsf{V}_s S$, then $g_s(u,v)=0$ for all $v \in \sT_s M$.  Since $\sT_s S \subset \sT_s M$, we have $g_s(u,v)=0$ for all $v \in \sT_s S$.  Thus, $u \in \ker(g_s)$, and so as $s$ is arbitrary, we establish that  $\Vect_\mathsf{V}(S)\subset \ker(g_S)$.
\item \label{pro:2} Conversely, suppose $u \in \ker(g_s) \subset \sT_s S$. By definition, $g_s(u, w) = 0$ for all $w \in \sT_s S$. Because $\mathsf{V}_s M \subset \mathrm{T}_s S$, we can decompose the ambient tangent space as  $\sT_s M = \sT_s S + \mathsf{H}_s M$. Since $g$ is strictly positive definite on $\mathsf{H}_s M$ and vanishes when evaluated against $\mathsf{V}_s M$, any horizontal component of $u$ transverse to $\sT_s S$ would violate the non-degeneracy of $g$ on the horizontal distribution. Thus, the kernel of $g_s$ is contained within the intersection of $\ker(g_s)$ and $\sT_s S$, which simplifies to $\mathsf{V}_s M \cap \sT_s S = \mathsf{V}_s S$. As $s$ is arbitrary, we establish that $\ker(g_S) \subset \Vect_\mathsf{V}(S)$. 
\end{enumerate} 
Together \eqref{pro:1} and \eqref{pro:2} establishes  that $\ker(g_S )=\Vect_\mathsf{V}(S)$.
\par 
Since the Carroll vector field $\kappa$ is tangent to $S$, evaluating $\tau_S$ on $\kappa|_S$ yields $(\iota^*\tau)(\kappa) = \tau(\kappa) = 1$. Since $\ker(g_S) = \Span\{\kappa|_S\}$, $\tau_S$ is a  clock form  $S$.  \par 
Finally, because $P|_S$ is symplectic, the restricted  map $P^\#_{S,s} : \sT_s^* S \rightarrow \sT_s S$ is an isomorphism, for all $s \in S$. Thus, the compatibility condition $\ker_s(g_S) \subset \textrm{Im}(P^\#_{S,s})$ is automatically satisfied for all points $s \in S$. 
\end{proof}
Hamiltonian vector fields on a Carrollian--Poisson manifold are defined as standard, i.e., given a function $f \in C^\infty(M)$, the associated Hamiltonian vector field is $X_f := P^\#(\rmd f)$.
\begin{theorem}\label{trm:LocalHam}
Let $(M, g , \tau, P)$ be a Poisson-stationary Carrollian--Poisson bundle.   Then, the fundamental (Carroll) vector field $\kappa \in \ker(g)$ is locally Hamiltonian on any symplectic leaf $S$ that is of non-zero dimension. 
\end{theorem}
\begin{proof}
Via Proposition \ref{prop:CarSymBun}, we know that given any symplectic leaf $S \subset M$ with non-zero dimension,  $\pi|_S : S \rightarrow \Sigma_S$ is a Carrollian--Symplectic bundle. Moreover, the symplectic structure, $\omega_S$ is equivariant, that is, $\mathcal{L}_{\kappa_S} \omega_S =0$, where we define $\kappa|_S := \kappa_S \in \Vect_\mathsf{V}(S)$. Thus, as $\omega_S$ is closed, we observe that
$$\mathcal{L}_{\kappa_S} \omega_S = \rmd (i_{\kappa_S} \omega_S ) + i_{\kappa_S}(\rmd \omega_S) = \rmd (i_{\kappa_S} \omega_S ) = 0\,.$$  
Then, via the Poincar\'e Lemma,  for every point $s \in S$, there exists an open subset $U \subseteq S$ such that $i_{\kappa_S} \omega_S|_U = \rmd f$, for some  function $f \in C^\infty(U)$.  This is algebraically identical to $\kappa_S = P_S^\#(\rmd f)$. Thus, $\kappa$ is locally Hamiltonian on any symplectic leaf that is not zero-dimensional.  
\end{proof}
%
%
\subsection{Further Examples of Carrollian--Poisson bundles}\label{subsec:MoreExamples}
We proceed to give further examples of Carrollian--Poisson bundles, including physically motivated examples. The approach to constructing examples is to lift Killing vector fields, thus the constructions generalise directly to many other Carrollian bundles. The method employed does not exhaust all classes of examples. 
\begin{example}\label{exa:BTZ}
Consider the stationary rotating BTZ metric \cite{Bañados:1992} on $\R^2 \times S^1$, which we write using co-rotating angular coordinates $(v, r, \varphi)$ as (in geometric units $G = c = 1$)
$$\rmd s^2 = - N^2 \, \rmd v^2 + 2 \rmd v \rmd r  + r^2 \big(\rmd \varphi +(\Omega_H + N^\phi) \rmd v \big)^2\,,$$
where 
$$N^2(r) = - M + \frac{r^2}{L^2} + \frac{J^2}{4 r^2}\,, \qquad N^\phi(r) = - \frac{J}{2 r^2}\,, \qquad \Omega_H = - N^\phi(r_+) = \frac{J}{2 r^2_+}\,,$$
and the physical parameters are the mass $M$, angular momentum $J$ and the AdS radius $L$. 
The radius of the event horizon is defined by the largest root of $N^2(r_+)=0$.  $\Omega_H$ is the angular velocity of the horizon.  At the horizon $r = r_+$, $N^2(r_+)=0$ and $\Omega_H + N^\phi(r_+)=0$ and the Carrollian metric on the horizon  $\mathcal{H}^+ \simeq \R \times S^1$ is 
$$g = \rmd \varphi  \otimes \rmd \varphi \, r^2_+\,.$$
Note that this metric is simply $g = \pi^* h \, r^2_+$, where $h$ is the standard round metric on $S^1$. Thus, $(\mathcal{H}^+, g , \tau)$  is a Carrollian bundle where we choose an arbitrary clock form $\tau$. We then define $P = r^{-1}_+ \, \kappa \wedge v^\mathsf{H}$, where $v \in \Vect(S^1)$ is the unique nowhere vanishing that satisfies $h(v,v) =1$ and $\Omega(v) >0$, where $\Omega$ is the canonical volume form associated with $h$. As all degree $3$ multivector fields on $\R \times S^1$ vanish, $P$ is Poisson. We pick a dual basis of $\Omega^1_\mathsf{H}(\mathcal{H}^+)$, which we denote as $\zx$.  Then $P^\#(\tau) = r^{-1}_+ \, v^\mathsf{H}$ and $P^\#(\zx) = - r^{-1}_+ \, \kappa$.  Thus, $\mathrm{Im}(P^\#) = \Span \{\kappa, v^\mathsf{H}  \}$, which implies that we have a Carrollian--symplectic bundle.  Clearly, the compatibility condition $\ker_m(g) \subset \mathrm{Im}(P_m^\#)$ is satisfied for all points $m \in\mathcal{H}^+ $. We have $\mathcal{L}_\kappa P = r_+^{-1}\, \kappa \wedge [\kappa, v^\mathsf{H}]$.  As $v^\mathsf{H}$ is projectable, $[\kappa, v^\mathsf{H}]$ is vertical, and thus $\mathcal{L}_\kappa P=0$, and so via Theorem \ref{trm:LocalHam}, $\kappa$ is locally Hamiltonian, as the unique symplectic leaf is the entirety of the  manifold $\mathcal{H}^+$.
\end{example}
\begin{example}\label{exa:KerBH}
The Kerr (outer) horizon is $\mathcal{H}^+\Diffeo \R \times S^2$.   For details of the Carrollian structure, see Marsot et al. \cite[Section III]{Marsot:2022}. We employ Eddington--Finkelstein-like coordinates $(v, \theta, \varphi)$, where $v$ is the null time, i.e., it parametrises the integral curves of the null generators,  and $\varphi$ is the angular coordinate that co-rotates with the horizon. Here $0 \leq \theta \leq \pi$ and $0 \leq \varphi < 2 \pi$.   As standard, we define (in geometric units $G = c = 1$)
$$r_+ = M + \sqrt{M^2 +a^2}\,, \qquad \Omega_{H} = \frac{a}{r^2_+ + a^2}\,,\qquad  \Sigma_+ = r^2_+ + a^2 \cos^2(\theta)\,,$$ 
where $M$ is the mass of the black hole and $a$ is the spin parameter.  The degenerate metric on $\mathcal{H}^+$ is 
$$g = \rmd \theta \otimes \rmd \theta \,  \Sigma_+ + \rmd \varphi \otimes \rmd \varphi \, \frac{(r_+^2 + a^2)^2}{\Sigma_+} \, \sin^2 (\theta)\,.$$ 
Clearly $\ker(g) = \Span \{ \partial_v\} = \Vect_{\mathsf{V}}(\mathcal{H}^+) $. \par 
The Kerr solution is topologically $\widetilde{M} \Diffeo \R \times \R \times S^2$. Let $X \in \Vect(\widetilde{M})$ be the rotational Killing vector field; this can be chosen canonically such that the integral curves have period $2 \pi$ and   the direction fixed to align with the black holes angular moment vector.  We then define $u:= X|_{S^2}$. Then via chosen clock form we define $u^\mathsf{H} \in \Vect_\mathsf{H}(\mathcal{H}^+)$.  We then define
$$P :=  \kappa \wedge u^\mathsf{H}\,,$$
Observe that $\SN{P,P} = -2 \,  \kappa \wedge u^\mathsf{V} \wedge [\kappa, u^\mathsf{H} ]=0$, as $u^\mathsf{H}$ is projectable, $[\kappa, u^\mathsf{H}]$ is vertical.  Note that $P$ is well-defined and is of rank $2$ everywhere except at the poles where $P$ vanishes (this can be checked using local coordinates). We then define $\mathcal{H}^+_0 := \mathcal{H}^+ \setminus \{N,S\}$, which is again a smooth manifold.  Using the metric, we can construct a  basis of $\Omega^1_\mathsf{H}(\mathcal{H}_0^+)$, which we denote as  $\{\zx, \eta \}$, where $\zx(u^\mathsf{H})= 1$ and $\eta(u^\mathsf{H})=0$. We observe that $P_m^\#(\tau_m) =  u_m^\mathsf{H}$ and $P_m^\#(\zx_m) = -  \kappa_m$, and $P_m^\#(\eta_m)=0$, provided $m \in \mathcal{H}_0^+$. Thus, away from the poles $\textrm{Im}(P_m^\#) = \Span \{\kappa_m, u_m^\mathsf{H} \}$, demonstrating the compatibility condition $\ker_m(g) \subset \mathrm{Im}(P_m^\#)$. At the poles, the Poisson structure is trivial, and thus the anchor map is a zero map.  We have $\mathcal{L}_\kappa P =  \kappa \wedge [\kappa, u^\mathsf{H}]=0$,  as $u^\mathsf{H}$ is projectable, which implies that  $[\kappa, u^\mathsf{H}]$ is vertical.  Thus, via Theorem \ref{trm:LocalHam}, $\kappa$ is locally Hamiltonian everywhere on $\mathcal{H}^+$ except at the poles.
\end{example}
\begin{example}\label{exa:FutNullInf}
Consider future null infinity  $\mathscr{I}^+ \cong \mathbb{R} \times S^{d-2}$, which we view as a Carrollian bundle $(\mathscr{I}^+, g= \pi^* h, \tau)$, where $h$ is the standard spherical metric on $S^{d-2}$, and the we define a clock form $\tau$.  The Lie algebra of Killing vector fields on $S^{d-2}$ is $\mathfrak{so}(d-1)$. Within the Killing vector fields, we have the Cartan Lie subalgebra, which is the largest abelian Lie subalgebra of the Killing vector fields. The generators of the Cartan Lie subalgebra we denote as $L_i$ with $i = 1, 2 \cdots  [(d-1)/2]=:k$. We can then form the Hopf-like vector field $L := \sum_{j=1}^k L_j$. Using the clock form, we define  
$$P =  \kappa \wedge L^\mathsf{H}\,.$$ 
As $L^\mathsf{H}$ is projectable, $[\kappa, L^\mathsf{H}]$ is vertical, implying $\SN{P,P}=0$.  Thus, $P$ is a Poisson structure on $\mathscr{I}^+$.  By the Hairy Ball Theorem, if $d$ is even, $S^{d-2}$ is an even-dimensional sphere and $L$ vanishes at isolated poles, meaning the rank of $P$ drops to zero at those two points. Conversely, if $d$ is odd, $S^{d-2}$ is an odd-dimensional sphere and $L$ is a nowhere-vanishing Hopf-like vector field, meaning $P$ is regular and of rank $2$ everywhere. We define the  set $\{ N,S\} := \{m \in \mathscr{I}^+ ~~| ~~ L^\mathsf{H}_m =0 \}$, which is empty in the case of odd dimensional spheres.   Then defining $\mathscr{I}^+_0 := \mathscr{I}^+ \setminus \{N,S\}$.  We then find (using the metric) a frame for $\Omega^1_\mathsf{H}(\mathscr{I}^+_0)$, which we denote as $\{\zx, \eta_i \}$, such that $\zx(L^\mathsf{H})=1$ and $\eta_i(L^\mathsf{H})=0$.  Then $P^\#(\tau) = L^\mathsf{H}$, $P^\#(\zx)= - \kappa$, and $P^\#(\eta_i)=0$. Thus, $\textrm{Im}(P^\#) = \Span\{\kappa, L^\mathsf{H} \}$. So, for all points of $\mathscr{I}^+_0$ the compatibility condition $\ker_m(g) \subset \mathrm{Im}(P^\#_m)$ holds.  We have $\mathcal{L}_\kappa P =  \kappa \wedge [\kappa, L^\mathsf{H}]=0$,  as $L^\mathsf{H}$ is projectable, which implies that  $[\kappa, L^\mathsf{H}]$ is vertical.  Thus, via Theorem \ref{trm:LocalHam}, $\kappa$ is locally Hamiltonian everywhere on $\mathscr{I}^+$  in the odd dimensional case,  and in the even dimensional case, everywhere except at the poles.
\end{example}
\begin{example}\label{exa:Torus}
Let $(T^n, h)$ be the flat $n$-torus.  Note, there are $n$-commuting killing vector fields and via rescaling we can always construct a frame of $\Vect(T^n)$, $\{v_1, \cdots , v_n\}$ consisting of orthonormal oriented Killing vector fields. That is, $h(v_i, v_j) = \delta_{ij}$ and $\Omega(v_i)> 0$, where $\Omega$  is the canonical volume form associated with $h$.  Note, that all such frames are isomorphic via the orthogonal group $O(n)$. We can then form the trivial principal $\R$-bundle $M := \R\times T^n \stackrel{\pi}{\rightarrow} T^n$.  By defining $g := \pi^*h$ and choosing a closed equivariant clock form $\tau$ (which is equivalent to a flat principal connection),  we construct the Carrollian bundle $(\R \times T^n, g, \tau)$.  Let us set $v := v_1 + v_2 + \cdots v_n$,  then 
$$P = \kappa \wedge v^\mathsf{H}\,, $$
is a Poisson structure on $\R \times T^n$ as $[\kappa, v^\mathsf{H}]=0$. Picking a dual basis of $\Omega^1_\mathsf{H}(\R \times T^n)$, which we denote as $\{\zx_i, \cdots \zx_n \}$, we observe that $P^\#(\tau) = v^\mathsf{H}$ and $P^\#(\zx_i) = -\kappa$. Thus, $\mathrm{Im}(P^\#) = \Span \{\kappa, v^\mathsf{H} \}$. Clearly, $\ker_m(g) \subset \mathrm{Im}(P_m^\#)$, for all points $m \in \R \times T^n$, and so $(\R \times T^n, \pi^*h, \tau, P)$ is a Carrollian--Poisson bundle. \par 
Let us consider a modified structure 
$$P = \kappa \wedge v^\mathsf{H} + \omega \sum_{i<j} v^\mathsf{H}_i \wedge v^\mathsf{H}_j \,,$$
where $\omega\in \R_{\geq 0}$  is a frequency. As we have employed  a closed equivariant clock, i,e., we employ a flat principal connection, we have $[\kappa, v^\mathsf{H}] = 0$ and $[v^\mathsf{H}_i, v^\mathsf{H}_j]=0$, and so $\SN{P,P}=0$.   We then observe that 
$$P^\#(\tau) = v^\mathsf{H}\,, \qquad P^\#(\zx_k) = - \kappa + \omega \big(\sum_{j>k} v^\mathsf{H}_j -  \sum_{i < k} v^{\mathsf{H}}_i \big)\,.$$
If this structure is to define a Carrollian--Poisson structure then there exists an $\alpha = \zx_k c^k$, with $c^k$, possibly functions, such that $P^\#(\alpha) = \kappa$.  To find such a horizontal one-form, let us write $P = P_0 + \omega P_1$. Then we require 
\begin{enumerate}[(1)]
\item $P_0^\#(\alpha) = \kappa$, \label{eqn:conP0}
\item $P_1^\#(\alpha)=0$. \label{eqn:conP1}
\end{enumerate}
For $n$ even, $P_1$ is invertible when restricted to horizontal one-forms, and thus the only solution to  \eqref{eqn:conP1} is $\alpha =0$.  However, this cannot be a solution of \eqref{eqn:conP0}.  Thus, we require $\omega =0$ in this case. \par 
For $n$ odd,  we can set $\alpha = \sum_{k=1}^n  (-1)^k \, \zx_k$ as a solution to \eqref{eqn:conP0}; one observes that  $ \sum_{k=1}^n  (-1)^k  = -1$.  We need to check    \eqref{eqn:conP1}, i.e.,
$$S = \sum_{k=1}^n (-1)^k \, \big(\sum_{j>k} v^\mathsf{H}_j -  \sum_{i < k} v^\mathsf{H}_i \big) \stackrel{?}{=} 0\,.$$
Let us isolate the coefficient  for $v^\mathsf{H}_m$ $(1 \leq m \leq n)$
\begin{align*}
C_m &= \sum_{k=1}^{m-1}(-1)^k - \sum_{k= m+1}^n (-1)^k = \sum_{k=1}^{m-1}(-1)^k - \left( -1 - (-1)^m - \sum_{k=1}^{m-1}(-1)^k \right) \\
&= 2 \, \sum_{k=1}^{m-1}(-1)^k + 1 + (-1)^m \,.
\end{align*}
If $m$  is even, then $m-1$ is odd and so $C_m = -2 + 1+1 =0$.  If $m$ is odd, then $m-1$ is even and so $C_m = 0+1 -1 =0$. As $C_m = 0$ for all $m$, $S = 0$.\par
Thus, $P = \kappa \wedge v^\mathsf{H} + \omega \sum_{i<j} v^\mathsf{H}_i \wedge v^\mathsf{H}_j$ defines a Carrollian--Poisson structure on Carrollian bundle $(\R \times T^n, g, \tau)$ with $\omega \neq 0$ provided $n$ is odd and the clock form $\tau$ is closed and equivariant. Moreover, we observe that in this case $\mathrm{Im}(P^\#) = \Span\{\kappa, v_1^\mathsf{H}, \cdots v_n^\mathsf{H} \}$ and so we have a Carrollian--Symplectic bundle.\par 
As $[\kappa, v_i^\mathsf{H}] =0$, $\mathcal{L}_\kappa P =0$, and so via  Theorem \ref{trm:LocalHam}, $\kappa$ is locally Hamiltonian. Note that a different choices of frames using orthonormal oriented Killing vector fields isomorphic Poisson structures. 
\end{example}
%
%
\subsection{Compatible Contravariant Connections}\label{subsec:CompContConn}
To define geometric kinematics, we require the notion of a connection so that geodesics can be defined. In the current situation, we wish the connection to take into account the Poisson structure and so we will work with contravariant or Vaisman connections as commonly found in Poisson geometry. \par
Recall that a \emph{contravariant connection} or \emph{Vaisman connection} on a Poisson manifold is an $\R$-linear map $D : \Omega^1(M)\times \Omega^1(M) \rightarrow \Omega^1(M)$ that satisfies  (see \cite{Fernandes:2000,Vaisman:1991})
\begin{enumerate}[i)]
\item $D_{f\, \alpha} \beta = f \, D_\alpha \beta$,
\item $D_\alpha f \, \beta = P^\#(\alpha)f \, \beta + f \, D_\alpha \beta$,
\end{enumerate}
for all $f \in C^\infty(M)$ and $\alpha, \beta \in \Omega^1(M)$. That is, a Vaisman connection is a Lie algebroid connection on the cotangent Lie algebroid of a Poisson manifold. The \emph{torsion tensor} of a Vaisman connection is defined as 
$$T_D(\alpha, \beta) := D_\alpha \beta - D_\beta \alpha - [\alpha, \beta]\,,$$
for all $\alpha, \beta \in \Omega^1(M)$. A Vaisman connection is said to be \emph{torsion-free} if $T_D=0$. 
  \par  
Contravariant connections can be extended to act on vector fields and and tensors using standard methods. In particular, $D_\alpha X$ is defined via
$$\beta(D_\alpha)X = P^\#(\alpha)\big( \beta(X) \big)- \big(D_\alpha \beta \big)(X)\,.$$
Similarly, we have 
\begin{subequations}
\begin{align}
\big (D_\alpha g\big)(X,Y)& = P^\#(\alpha)\big(g(X,Y) \big) - g(D_\alpha X, Y) - g(X, D_\alpha Y)\,, \\ 
\big(D_\alpha P\big)(\beta, \gamma) &= P^\#(\alpha)\big( P(\beta, \gamma)\big) - P(D_\alpha \beta, \gamma) - P(\beta, D_\alpha \gamma )\,.
\end{align}
\end{subequations}
\begin{definition}\label{def:ContraMetandPoissonConn}
A Vaisman connection is said to be
\begin{enumerate}[1)]
\item \emph{metric-compatible} if $ D g=0$; and
\item \emph{clock-compatible} if $D\tau = 0$.
\end{enumerate}
We refer to a Vaisman connection that is metric-compatible and clock-compatible  as a \emph{Carroll--Vaisman connection}.
\end{definition}
\begin{observation}
From \cite[Corollary 2.58]{Bruce:2026a}, we know that on any Carrollian bundle metric-compatible affine connections always exist (not unique). Thus, on any Carrollian-Poisson bundle, we may define $D_\alpha  := \nabla_{P^\#(\alpha)}$, and so metric-compatible Vaisman connections always exist. 
\end{observation}
\begin{remark}
It is well known that on any Poisson manifold, Poisson-compatible Vaisman connections always exists and are not unique.  Thus, any Carrollian--Poisson bundle may be equipped with a Poisson-compatible Vaisman connection, i.e., a contravraiant connection such that $DP=0$. However, imposing this constraint  puts constraints on the geometry that forbids physically interesting examples.  We will not comment on this further.
\end{remark}
\begin{definition}
A \emph{strong Carrollian--Poisson bundle} is a quintuple $(M, g , \tau, P , D)$, where $(M,g, \tau, P)$ is a Carrollian--Poisson bundle and $D$ is a  metric-compatible and clock-compatible Vaisman  connection.
\end{definition}
\begin{lemma}\label{lem:DPres}
Let $(M, g, \tau, P, D)$ be a  strong Carrollian--Poisson bundle. Then
\begin{enumerate}[i)]
\item  $D_\alpha X_\sV \in \Vect_\sV(M)$; and
\item $D_\alpha X_\sH \in \Vect_\sH(M)$,
\end{enumerate}
for all $\alpha \in \Omega^1(M)$.
\end{lemma}
\begin{proof}\
\begin{enumerate}[i)]
\item As $D$ is metric-compatible we have
$$P^\#(\alpha)\big(g(X,Y) \big) - g(D_\alpha X, Y) - g(X, D_\alpha Y)=0\,,$$ 
for all $\alpha \in \Omega^1(M)$ and $X,Y \in \Vect(M)$. Then setting  $X = X_\sV \in \Vect_\sV(M) = \ker(g)$, we obtain $g(D_\alpha X_\sV, Y)=0$. As $Y \in \Vect(M)$ is arbitrary, it must be the case that  $D_\alpha X_\sV \in \ker(g)$ for all $\alpha$.
\item Starting with 
$$\beta(D_\alpha X)= P^\#(\alpha)(\beta(X)) - (D_\alpha \beta)(X)\,,$$
and setting $\beta = \tau$, together with $D$ being clock-compatible, we have $\tau(D_\alpha X_\sH)= P^\#(\alpha)(\tau(X_\sH))=0$ as $\ker(\tau)= \Vect_\sH(M)$.  Thus, $D_\alpha X_\sH$ is in the kernel of $\tau$, and so $D_\alpha X_\sH \in \Vect_\sH(M)$ for all $\alpha$.
\end{enumerate} 
\end{proof}
Lemma \ref{lem:DPres}, tells us that on any strong Carrollian--Poisson bundle, the Vaisman connection restricts to $\ker(g)= \Vect_\sV(M)$ and $\Vect_\sH(M)$, and so respects the Carrollian structure.
\begin{lemma}\label{lem:MetComCont}
Let $(M, g, \tau, P)$ be a Carrollian--Poisson manifold. Then any metric-compatible Vaisman connection can be brought to the form 
$$D_\alpha X :=  \nabla_{P^\#(\alpha)}X - \Phi_\alpha(X)\, \kappa - \Omega_\alpha(X_\mathsf{H}) \,,$$
where 
\begin{enumerate}[i)]
\item $\nabla : \Vect(M) \times \Vect(M) \rightarrow \Vect(M)$ is a metric compatible affine connection;
\item $\Phi :  \Omega^1(M) \times \Vect(M) \rightarrow C^\infty(M)$ is a $C^\infty(M)$-linear map;
\item $\Omega : \Omega^1(M) \times \Vect_\mathsf{H}(M) \rightarrow \Vect_\mathsf{H}(M)$ is a $C^\infty(M)$-linear map that is skewsymmetric  with respect to $g$,
\end{enumerate}
using the clock-form to define $X = X_\mathsf{V}+ X_\mathsf{H}$, and $\kappa \in \Vect_\sV(M)$ is the Carroll vector field.
\end{lemma}
\begin{proof}
From Definition \ref{def:ContraMetandPoissonConn}
\begin{align*}
\big (D_\alpha g\big)(X,Y) &= P^\#(\alpha)\big(g(X,Y) \big) - g(D_\alpha X, Y) - g(X, D_\alpha Y)\\
& = g(\nabla_{P^\#(\alpha)}X,Y) + g(X, \nabla_{P^\#(\alpha)}Y) - g(D_\alpha X, Y) - g(X, D_\alpha Y)\\
& = g\big( (\nabla_{P^\#(\alpha)}- D_\alpha) X , Y  \big) + g \big(X, (\nabla_{P^\#(\alpha)}- D_\alpha)Y \big)\,,
\end{align*}
where we have used  \cite[Corollary 2.58]{Bruce:2026a} which established the existence of metric compatible affine connections on Carrollian manifolds.  Setting $A_\alpha = \nabla_{P^\#(\alpha)} - D_\alpha$, we have the condition
$$g(A_\alpha X,Y) + g(X, A_\alpha Y) =0\,,$$
for $D$ to be metric compatible.  That is, the operator $A_\alpha$ is skewsymmetric with respect to $g$. Exploring the degrees of freedom we observe that
\begin{enumerate}[i)]
\item If $X = X_\mathsf{V}$ and $Y = Y_\mathsf{V}$, then the skewsymmetric condition is vacuously satisfied.
\item If $X = X_\mathsf{V}$ and $Y_\mathsf{H}$, then $g(A_\alpha X_\mathsf{V}, Y_\mathsf{H}) =0$, and thus, $A_\alpha X_\mathsf{V} \in \Vect_\mathsf{V}(M)$. 
\item If $X = X_\mathsf{H}$ and $Y = Y_\mathsf{H}$,  then $g(A_\alpha X_\mathsf{H},Y_\mathsf{H}) + g(X_\mathsf{H}, A_\alpha Y_\mathsf{H}) =0$, and thus $A_\alpha$ is a skewsymmetric operator on $(\Vect_\mathsf{H}(M), g)$.
\end{enumerate}
Thus, we have $A_\alpha X = \Phi_\alpha (X)\kappa + \Omega_\alpha(X_\mathsf{H})$, where $\Phi_\alpha(X) \in C^\infty(M)$ is unconstrained, and $\Omega_\alpha : \Vect_\mathsf{H}(M) \rightarrow \Vect_\mathsf{H}(M)$  is skewsymmetric with respect to $g$. From the definition of $A_\alpha$, we have $D_\alpha X = \nabla_{P^\#(\alpha)}X - \Phi_\alpha(X)\, \kappa - \Omega_\alpha(X_\sH)$, as required.
\end{proof}
We can then present any metric-compatible Viasman connection as
\begin{equation}\label{eqn:MetComContCon}
D_\alpha \beta  = \nabla_{P^\#(\alpha)}\beta + \beta(\kappa)\Phi_\alpha + B_\alpha^*\beta \,,
\end{equation}
where  $B_\alpha := \Omega_\alpha \circ \mathrm{p}_\mathsf{H} : \Vect(M) \rightarrow \Vect(M)$, with $\mathrm{p}_\mathsf{H} : \Vect(M) \rightarrow \Vect_\mathsf{H}(M)$ being the canonical projection to horizontal vector fields.  Note $\textrm{Im}(B_\alpha ) \subset \Vect_\mathsf{H}(M)$. The dual operator 
$$B_\alpha^* : \Omega^1(M) \rightarrow \Omega^1(M)\,,$$
is defined as $B_\alpha^*\beta :=  \beta \circ B_\alpha$.
\begin{theorem}\label{the:MetClocCon}
Let $(M, g, \tau, P)$ be a Carrollian--Poisson bundle, then Vaisman connections that are both metric-compatible and  clock-compatible always exists. 
\end{theorem}
\begin{proof}
From Lemma \ref{lem:MetComCont} (and \eqref{eqn:MetComContCon}) we observe that, for any metric-compatible Vaisman connection 
$$D_\alpha \tau =  \nabla_{P^\#(\alpha)}\tau + \tau(\kappa)\, \Phi_\alpha + B_\alpha^* \tau =  \nabla_{P^\#(\alpha)}\tau +  \Phi_\alpha\,.$$
Thus, if $D$ is to be clock compatible then $\Phi_\alpha = - \nabla_{P^\#(\alpha)}\tau$ is required.  We may then construct such a Vaisman connection by selecting a metric compatible affine connection $\nabla$ (which always can be done) and  defining $\Phi_\alpha := - \nabla_{P^\#(\alpha)}\tau$ and keeping $\Omega$ arbitrary. Thus, metric-compatible and clock-compatible Vaisman connections can always be constructed.
\end{proof}
We then deduce that any metric-compatible and clock-compatible connection has as decomposition of the form 
\begin{equation}
D_\alpha \beta  = \nabla_{P^\#(\alpha)}\beta - \beta(\kappa)\,\nabla_{P^\#(\alpha)}\tau + B_\alpha^*\beta \,,
\end{equation}
with respect to a metric-compatible affine connection.\par
The above analysis demonstrates that on any Carrollian--Poisson bundle, selecting a metric-compatible and clock-compatible Vaisman connection is always possible. That is, any Carrollian--Poisson bundle can be rendered a strong Carrollian--Poisson bundle, albeit non-canonically.  
\begin{example}\label{exa:FutNullInftD}
Continuing Example \ref{exa:FutNullInf}, consider future null infinity  $\mathscr{I}^+ \cong \mathbb{R} \times S^{d-2}$, which is  Carrollian--Poisson bundle $(\mathscr{I}^+, g= \pi^* h, \tau, P)$, where $h$ is the standard spherical metric on $S^{d-2}$, the clock form  $\tau$ is arbitrary, and $P =  \kappa \wedge L^\mathsf{H}$ is built from the Cartan Lie subalgebra of Killing vector fields.  Let $\nabla^0$ be the Levi-Civita connection on $S^{d-2}$ associated with the metric $h$. As $\sH M \cong \pi^*(\sT S^{d-2}) $, the Levi-Civita connection can be pullbacked to construct a horizontal connection on $\Vect(\mathscr{I}^+)$, which we will denote as $\nabla^\sH$.  Note that by construction $\nabla^\sH_\kappa =0$.  We then have the `lifted' affine connection defined as 
$$\nabla_X Y :=  \nabla^\sH_X Y_H + X(\tau(Y))\, \kappa\,.$$ 
As $\kappa \in \ker(g)$, only the horizontal projection of $\nabla_X Y$ important for metric compatibility. Because $\nabla^\sH$ is the pullback of the Levi-Civita connection and $g = \pi^* h$, we automatically have metric-compatibility. Taking the dual affine connection, and then defining 
$$D_\alpha \beta :=  \nabla_{P^\#(\alpha)}\beta =\nabla^\sH_{P^\#(\alpha)}\beta_\sH + P^\#(\alpha)(\beta(\kappa))\, \tau\,, $$
we obtain a (canonical) Carroll--Vaisman connection.  One can directly check clock-compatibility using the fact that $\tau$ is vertical and that $\tau(\kappa)=1$.  Note we still have freedom in choosing $B_\alpha$, and here we simply set it to the zero map.  Thus, we have constructed the strong  Carrollian--Poisson bundle $(\mathscr{I}^+, g, \tau, P, D)$.
\end{example}
\begin{remark}
The mechanism of lifting a Levi-Civita connection as demonstrated in Example \ref{exa:FutNullInftD}, directly generalises to Carrollian--Poisson bundles where the metric is the pullback of a Riemannian metric on the base manifold $\Sigma$. That is, whenever $g$ is equivaritant.  Thus, the other examples in Subsection \ref{subsec:MoreExamples} may similarly be equipped with a Carroll--Vaisman connection.
\end{remark}
%
%
\subsection{Geometric Kinematics on Carrollian--Poisson Bundles}\label{subsec:GeoKin}
Let $I = [s_0, s_1] \subset \R$ be a closed interval, and let $\pi^{\sT^*M}_M : \sT^*M \rightarrow M$ denote the projection to $M$.  A (smooth) curve $c : I \rightarrow  \sT^* M$ over $\gamma := \pi^{\sT^*M}_M \circ c : I \rightarrow M$, is said to be an \emph{admissible curve} if  $P^\#\big(c(s)\big) = \dot{\gamma}(s)$, for all $s \in I$.  We can further project the curve using $\pi : M \rightarrow \Sigma$, and so we define the \emph{spatial curve} as $\widehat{\gamma}:= \pi \circ \gamma : I \rightarrow \Sigma$. \par 
Note that as the velocity of an admissible curve $\dot{\gamma}(s) = P^\#(c(s))$ for all $s \in I$  is tangent to a single symplectic leaf,    $\gamma(s)$ lies in the single symplectic leaf that contains $m_0 := \gamma(s_0)\in M$. 
\begin{definition}
Let $(M, g , \tau, P , D)$ be a strong Carrollian--Poisson bundle. An admissible curve $c :I \rightarrow \sT^* M$ is said to be a  \emph{contravariant geodesic} if $D_{c(s)} c(s) =0$ for all $s \in I$.  A contravariant geodesic is said to be a
\begin{enumerate}[i)]
 \item \emph{Carrollian contravariant geodesic} if the underlying spatial  curve is constant while the temporal  curve is non-constant;
\item \emph{Tachyonic contravariant geodesic} if the underlying spatial curve is non-constant; and
\item \emph{Singular contravariant geodesic} if the underlying spatial curve and temporal curve are constant.
\end{enumerate}
\end{definition}
\begin{remark}
The notions of Carrollian, tachyonic, and singular admissible curves are similarly defined. 
\end{remark}
The notion of a Carrollian contravariant geodesic captures the idea of `frozen motion', that is, massive particles being constrained to only move in the temporal direction and spatially static.  A tachyon can move spatially, and so the nomenclature tachyonic contravariant geodesics.  Singular contravariant geodesics are single points. For example, we can picture a particle completely frozen on a zero-dimensional symplectic leaf. \par
To set notation with the initial conditions of an admissible curve, we set $p := c(s_0) \in \sT^*_{m}M$, with $m := \gamma(s_0) \in M$, and $v := P^\#_{m}(p) \in \sV_m M \bigoplus \sH_m M$  (using the clock form).
\begin{proposition}\label{prop:IniConMin}
Let $(M, g, \tau, P, D)$ be a strong  minimal Carrollian--Poisson bundle, i.e, we have $P_\sH =0$. A contravariant geodesic $c: I \rightarrow \sT^* M$ over $\gamma := \pi^{\sT^*M}_M \circ c : I \rightarrow M$ is a Carrollian contravariant geodesic if and only if  $v:= P^\#_{m}(p)  \in \sV_m M$ and non-zero, and $P_m \neq 0$.
\end{proposition}
\begin{proof}
Let $c(s)$ be a Carrollian contravariant geodesic.  By definition $\widehat{\gamma}(s) = (\pi \circ \gamma)(s)$ is constant, and so $\dot{\widehat{\gamma}}(s)= \rmd\pi (\dot{\gamma}(s))=0$. As $\ker(\rmd \pi_{\gamma(s)}) = \sV_{\gamma(s)}M$, we observe that $\dot{\gamma}(s) \in \sV_{\gamma(s)}M$ for all $s \in I$.  Thus, if $c(s)$ is a Carrollian contravariant geodesic, then $v:= \dot{\gamma}(s_0) =P^\#_m (c(s_0)) \in \sV_m M$, which must be non-zero for the temporal component of the geodesic to be non-constant. If $P_m=0$, then the symplectic leaf  containing $m$ is the  single point $m$ (or equivalently $v =0$), and thus we cannot have a Carrollian curve, i.e., we have a contradiction and so $P_m \neq 0$. \par 
Assume $v \in \sV_m M$ at $m = \gamma(s_0)$ is non-zero, and   $P_m \neq 0$. The underlying spatial curve is constant if and only if the horizontal component $\gamma_\sH(s)$ of the velocity $\dot{\gamma}(s) = P_{\gamma(s)}^\#(c(s))$ vanishes for all $s \in I$. Consider the pairing $\langle c(s), \kappa\rangle$, and in particular, differentiating  this along the geodesic, we obtain 
\begin{align*}
\frac{\rmd}{\rmd s} \langle c(s), \kappa  \rangle &= \langle D_{c(s)}c(s), \kappa \rangle + \langle c(s) , D_{c(s)}\kappa\rangle\\
& = \psi(s)\, \langle c(s), \kappa\rangle\,, 
\end{align*}
where we have used $D_{c(s)} c(s)=0$ and  Lemma \ref{lem:DPres} to write $D_{c(s)} \kappa = \psi(s)\, \kappa$. As $v \in \sV_m M$ (at $m = \gamma(s_0)$), it must be the case that $c(s_0) \in \sH^*_m M$ due to the minimal condition and $P_m \neq 0$. Thus, we have a first-order differential equation subject to the initial condition $\langle c(s_0), \kappa \rangle=0$.  Then, via the Picard--Lindelöf theorem, we know that $\langle c(s), \kappa \rangle=0$ for all $s \in I$. This implies that $c(s) \in \sH^*_{\gamma(s)}M$ for all $s \in I$. Thus, $P_{\gamma(s)}^\#(c(s)) \in \sV_{\gamma(s)} M$ for all $s \in I$. Therefore, $\gamma_\sH(s) =0$ for all $s \in I$.
 \end{proof}
Proposition \ref{prop:IniConMin} establishes that on strong minimal Carrollian--Poisson bundles, Carrollian contravariant geodesics can be formulated as an initial value problem.  However, it is important to stress that this is not the case for non-minimal Carrollian--Poisson bundles.  Specifically, a Carrollian contravariant geodesic must satisfy the geometric constraint
\begin{equation}
c(s) \in \ker\big(P^\#_{\sH, \gamma(s)}\big)\,,
\end{equation} 
for all $s \in I$. It is not the case that specifying $c(s_0) \in \ker\big(P^\#_{\sH,m}\big)$ implies the geometric constraint, and so we do not have an initial value problem.\par 
In general, once we have selected a reference metric affine connection, the contravariant geodesic equation can be cast into the form 
\begin{subequations}
\begin{equation}
\nabla_{\dot{\gamma}(s)}c(s) = \langle c(s) , \kappa \rangle \nabla_{\dot{\gamma}(s)} \tau - B^*_{c(s)} c(s)\,.  
\end{equation}
If the Carrollian--Poisson bundle is minimal, then  Proposition \ref{prop:IniConMin} establishes that for Carrollian contravariant geodesics $\langle c(s), \kappa\rangle =0$. That is, $c(s) = c_\sH(s)$, i.e., the momentum is horizontal.  Thus, the geodesic equation  for Carrollian geodesics simplifies to
\begin{equation}
\nabla_{\dot{\gamma}(s)}c_\sH(s) =  - B^*_{c_\sH(s)} c_\sH(s)\,.
\end{equation}
\end{subequations}
\begin{example}
Continuing Example \ref{exa:FutNullInftD}, we equip $(\mathscr{I}^+, g , \tau, P)$, with the Carroll--Vaisman connection $D_\alpha \beta :=  \nabla^\sH_{P^\#(\alpha)}\beta_\sH + P^\#(\alpha)(\beta(\kappa))\, \tau$. For any admissible curve $c : I \rightarrow \sT^* \mathscr{I}^+$, the contravariant geodesic equation can be written as
$$\nabla^\sH_{\dot{\gamma}(s)} c_\sH(s) = - \dot{\gamma}(s)\langle  c(s), \kappa\rangle\, \tau\,.$$
Let us consider Carrollian contravariant geodesics, and find local solutions using coordinates $(t, \theta^i)$, where $\theta^i$ are  angular coordinates.  As we have no spatial motion, we select a point (not one of the poles) and set $\theta^i(s) =\theta^i_0$. Via Proposition  \ref{prop:IniConMin}, we know that $ \langle c(s), \kappa \rangle =0 $ for all $s \in I$.  Thus, the momentum is horizontal; we set $c(s) := c_\sH(s) = \rmd \theta^i c_i(s)$.  The contravariant geodesic equation simplifies to
$$\nabla^\sH_{\dot{\gamma}(s)} c_\sH(s) =0\,.$$
This implies that $c_i(s)$ do not change along the $\R$-fibres, thus we set $c_i(s) = p_i$ understood as the spatial momentum.  Then using the anchor, we have $P^\#(\rmd \theta^i p_i) = - (L^\sH)^i p_i\, \kappa$, and so $\dot{t}(s) = -  (L^\sH)^i(\theta_0)p_i$. Finally, the local Carrollian geodesic is given by 
$$t(s) = t_0 - \big ((L^\sH)^i(\theta_0)p_i  \big)\, s\, , \qquad \theta^i(s) = \theta^i_0\,.$$
\end{example}
\begin{example}
Continuing Example \ref{exa:BTZ}, consider the BTZ horizon $\mathcal{H}^+ \Diffeo \R \times S^2$, which we equip with coordinates $(v, \varphi)$. We will, for simplicity, choose the clock form $\tau = \rmd v$. The Poisson structure is then given by $P = r_+^{-1}\, \partial_v \wedge \partial_\varphi$. Similarly, to  the previous example, we employ the lift of Levi-Civita connection on $S^2$ associated with the (rescaled) flat metric; the Christoffel symbols are zero, simplifying the expressions. We parameterise admissible curves as $c(s) = \rmd v \,p_v(s) + \rmd \vartheta \,p_\varphi(s)$. Thus the velocity of a curve 
$$\dot{v}(s) = - p_\varphi(s) r_+^{-1}\,, \qquad \dot{\varphi}(s) = p_v(s) r_+^{-1}\,.$$
The contravariant geodesic equation is particularly simple 
$$\rmd v \, \dot{p}_v(s) + \rmd \varphi \,\dot{p}_\varphi (s) =0\,,$$
and thus, $p_v(s) = p_v^0$ and $p_\varphi(s) = p_\varphi^0$ (constant). Integrating the velocity produces a paramaterisation of contravariant geodesics as
$$v(s) = v_0 - (p_\varphi^0 r_+^{-1})\, s \,, \qquad \varphi(s) = \varphi_0 + (p_v^0 r_+^{-1})\, s\,.$$  
\begin{description}
\item[Carrollian] ($p_v^0=0$) \\
$$v(s) = v_0 - (p_\varphi^0 r_+^{-1})\, s \,, \qquad \varphi(s) = \varphi_0 \,,$$
\item[Tachyonic] ($p_v^0 \neq 0$)\\
\begin{enumerate}[i)]
\item \emph{Circular} ($p_\varphi^0 =0$) 
$$v(s) = v_0  \,, \qquad \varphi(s) = \varphi_0 + (p_v^0 r_+^{-1})\, s\,.$$
\item \emph{Helical} ($p_\varphi^0 \neq 0$)
$$v(\varphi) = v_0 - \left(\frac{p_\varphi^0}{p_v^0} \,(\varphi - \varphi_0)\right)\,.$$
\end{enumerate}
\end{description}
\end{example}
%
%
\section{Concluding Remarks}\label{sec:ConMark}
We have proposed the notion of a Carrollian--Poisson bundle as a Carrollian bundle equipped with a compatible Poisson structure and presented a study of their fundamental properties. The motivation for this structure is to find a quasi-classical framework to study aspects of `Carrollian noncommutative spaces'. The compatibility condition $\ker(g) \subset \mathrm{Im}(P^\#)$ translates to the existence of at least  one-form  $\alpha$ such that $P^\#(\alpha) = \kappa$, where $\kappa$ is the Carroll vector field.  This condition is weaker than insisting that the Carroll vector field be (locally or globally) Hamiltonian, i.e., asserting that $\alpha$ must be closed.   However, if the Poisson structure is independent of time, more mathematically put, equivariant under the principal $\R$-action, then $\kappa$ is locally Hamiltonian.  Thus, we expect physically relevant Carrollian--Poisson bundles are to be Poisson-stationary, i.e., $\mathcal{L}_\kappa P=0$.  We established the global decomposition $P = \kappa \wedge E_\mathsf{H} + P_\mathsf{H}$, and that with $\dim M \leq 3$, $P_\mathsf{H}=0$ automatically. The existence of Vaisman connections that are both  metric-compatible and clock-compatible on any Carrollian--Poisson bundle was proven. Once we have selected a Carrollian--Vaisman connection, as we term them, the geometric kinematics of test particles via geodesics can be studied. In particular, we have shown that Carrollian contravariant geodesics on strong minimal Carrollian--Poisson bundles can be posed as an initial value problem.   We therefore consider strong (minimal) Carrollian-Poisson bundles to represent an effective low-energy geometric model of a quantum Carrollian gravity background.  
\medskip 

\noindent \textbf{Future Directions:}\
\begin{enumerate}[i)]
\item Developing further examples of (strong) Carrollian--Poisson bundles, including those motivated by condensed matter physics and fluid dynamics is highly desirable (see \cite[Chapters 11--14]{Bagchi:2025} and references therein for a review of Carrollian physics in hydrodynamics and condensed matter physics). As well as aiding in the mathematical understanding of Poisson structures in Carrollian geometry, additional examples will strengthen the potential applications in physics. For example, contravariant geodesics that allow spatial motion are a natural feature of the strong Carrollian--Poisson bundles; we refer to such paths as tachyonic contravariant geodesics. While such curves seem very much at odds with the Carrollian limit $c\rightarrow 0$, they may arise when considering collective or emergent transport. In Carrollian hydrodynamics, while fundamental fluid particles are spatially fixed, acoustic and thermal transport currents may follow non-trivial paths. Similarly, in condensed matter systems with flat bands or fracton order, individual excitations are localized, but composite dipoles, for example, may  move freely along tachyonic trajectories.
\item Minimal, Poisson-stationary Carrollian--Poisson bundles $(M, g, \tau, P)$ can always be integrated to  produce a (global) symplectic Lie groupoid $\mathcal{G} \rightrightarrows M$ (the proof of which is outside the scope of the current paper).  In general, there will be Crainic--Fernandes monodromy group obstructions (see \cite{Crainic:2004}). However, even if the Carrollian--Poisson bundle is (globally) integrable, the Carrollian metric $g$ and clock form $\tau$ do not produce a standard Riemannian metric on $\mathcal{G}$, but instead are related to sub-Riemannian Lie groupoids and transverse metric structures. The axiomatic formulation of ``Carrollian--symplectic Lie Groupoids'' and their properties awaits exploration. 
\item While our motivation for Carrollian--Poisson bundles is deeply rooted in deformation quantisation, we have not pursued that programme at all here, and so many questions are open.  For instance, under what conditions does a Carrollian--Poisson bundle admit a formal star-product deformation quantisation such that the Carroll vector field $\kappa$ remains an inner derivation at all orders in $\hbar$? Furthermore, how does the compatibility condition $\ker_m(g) \subset \text{Im}(P^\#_m)$ constrain the representation theory of the corresponding quantised algebra $\mathcal{A}_\hbar$?
 \end{enumerate}
%
%

%
%

\end{document}